\documentclass[oneside,english,british]{amsart}

\usepackage{hyperref}

\hypersetup{linkcolor=red,citecolor=blue,filecolor=dullmagenta,urlcolor=darkblue} 

\newcommand{\ignore}[1]{}

\newcommand{\bb}{\mathbb}

\newcommand{\C}{\bb C} 

\newcommand{\Z}{\bb Z}

\newcommand{\R}{\bb R}
\newcommand{\N}{\bb N}

\newcommand{\Q}{\mathbb Q}

\newcommand{\cH}{\mathcal{H}}

\newcommand\norm[1]{\left\|#1\right\|}

\newcommand\abs[1]{\left|#1\right|}

\newtheorem{Theorem}{Theorem}
\newtheorem{Cor}[Theorem]{Corollary}
\newtheorem{question}[Theorem]{Question}
\newtheorem{Prop}[Theorem]{Proposition}
\newtheorem{Lemma}[Theorem]{Lemma}
\newtheorem*{lemma*}{Lemma}

\newtheorem*{theorem*}{Theorem}

\numberwithin{equation}{section}
\numberwithin{Theorem}{section}

\begin{document}
\title[Optimal density for polynomial maps]{Optimal density for values \\ of generic polynomial maps}
\author{Anish Ghosh, Alexander Gorodnik, and Amos Nevo} 
\address{School of Mathematics, Tata Institute of Fundamental research, Mumbai, India }
\email{ghosh@math.tifr.res.in}
\address{School of Mathematics, University of Bristol, Bristol UK }
\email{a.gorodnik@bristol.ac.uk}
\address{Department of Mathematics, Technion IIT, Israel}
\email{anevo@tx.technion.ac.il}

\thanks{The first author acknowledges support of UGC. The third author acknowledges support of ISF}

\begin{abstract}
We establish that the optimal bound for the size of the smallest integral solution of the Oppenheim Diophantine approximation problem 
$\abs{Q(x)-\xi}< \epsilon$ for a generic ternary form $Q$ is $\abs{x}\ll \epsilon^{-1}$. We also establish an optimal rate of density for the values of polynomials maps in a number of other natural problems, including the values of linear forms restricted to suitable quadratic surfaces, and the values of the polynomial map defined by the generators of the ring of conjugation-invariant polynomials on $M_3(\C)$. 

These results are instances of a general approach that we develop, which considers a rational affine algebraic subvariety of Euclidean space, invariant and homogeneous under an action of a semisimple Lie group $G$. Given a polynomial map $F$ defined on the Euclidean space which is invariant under a semisimple subgroup $H$ of the acting group $G$, consider the family of its translates $F\circ g$ by elements of the group. We study the restriction of these polynomial functions to the integer points on the variety confined to a large Euclidean ball. Our main results establish an explicit rate of density for their values, for generic polynomials in the family. This problem has been extensively studied before when the polynomials in question are linear, in the context of classical Diophantine approximation, but very little was known about it for polynomial of higher degree.  We formulate a heuristic pigeonhole lower bound for the density and an explicit upper bound for it, formulate a sufficient condition for the coincidence of the lower and upper bounds, and in a number of natural examples establish that they indeed match.
  Finally, we also establish a rate of density for values of homogeneous polynomials on homogeneous projective varieties. 

\end{abstract}

\maketitle
 
 \tableofcontents
 
\section{Introduction and statement of main results}\label{sec:intro}
\subsection{Introduction}
The present paper is devoted to establishing effective rates of density for values of generic polynomial maps computed at integral points. While the case of linear maps is a subject of the classical theory of Diophantine approximation, very little is known about polynomial maps of higher degrees. The present paper makes a systematic advance in this direction, including the derivation of some optimal results, which have few, if any, precedents.

One of the simplest (non-linear) instances of the deterministic form of this question 
is the quantitative density of values $Q(x)$, $x\in\mathbb{Z}^n$, 
for irrational indefinite quadratic forms
$Q$. Given $\xi\in \mathbb{R}$ and $\varepsilon>0$, one would like to establish
existence of $x\in\mathbb{Z}^n$ satisfying
$|Q(x)-\xi|<\varepsilon$ with an \emph{explicit} bound on the size of the vector $x$. This question was studied in \cite{BD,GM,LM} (we refer
to \S\ref{sec:quad0} below for a more elaborate discussion).
An analogous question can be asked for other polynomial maps.
For example, the density of values of linear maps on rational quadratic surfaces
was studied in \cite{S1} (see \S\ref{sec:lin} below).

Let us formulate this problem more generally.
Let ${\sf X}\subset \mathbb{A}^n$ be an affine algebraic variety defined over $\mathbb{Q}$ such that the set of its integral points ${\sf X}(\mathbb{Z})$ is Zariski dense in ${\sf X}$, and let $F=(F_1,\ldots,F_m):{\sf X}\to \mathbb{A}^m$ be a polynomial map defined over $\mathbb{R}$ such that
$F({\sf X}(\mathbb{R}))=\mathbb{R}^m$.
Provided that $F$ satisfies some irrationality assumptions, one might hope 
to show that the set of values
$F(x)$, $x\in {\sf X}(\mathbb{Z})$, is dense in $\mathbb{R}^m$.
We would like to explore quantitative aspects of the density of this set.
To state our question precisely, we fix an exponent $\kappa>0$,
and for $\xi\in \mathbb{R}^m$ and $\varepsilon>0$, we consider the system of inequalities
\begin{equation}\label{eq:main_ineq}
\|F(x)-\xi\|<\varepsilon,\quad \|x\|<\varepsilon^{-\kappa}\quad \hbox{ with $x\in {\sf X}(\mathbb{Z})$}, 
\end{equation}
where $\|\cdot \|$ denotes the maximum norm.
Analytic methods, which allow to establish
density  of values, usually can be also used to show the 
existence of $\kappa>0$ such that this system of inequalities
has solutions, but typically such $\kappa$ is far from optimal.
Naturally, the system \eqref{eq:main_ineq} cannot be solved for all $\xi$ if the exponent $\kappa$ is too small.
The lower bound $\kappa$ is expected to obey can be gleaned from the following heuristic argument. 

\noindent {\bf Pigeonhole Heuristics.} Suppose that  the polynomial map $F$ has degree $d$, and for some $a>0$ the following growth bound 
$$
|\{x\in {\sf X}(\mathbb{Z}):\,\, \|x\|<T\}|\ll T^a
$$
is satisfied for all sufficiently large $T$.
Since the values $F(x)$ with $\|x\|<T$ lie in a box 
of size $O(T^{d})$, it is natural to expect that a fixed bounded set contains at most 
$O(T^{a-md})$ of these values.
Then it follows from the pigeonhole principle that the set of values could be dense at most on the scale $O(T^{-(a-md)/m})$. Hence, in order to have solutions in
\eqref{eq:main_ineq} for all $\xi \in \mathbb{R}^m$ and all sufficiently small $\varepsilon$,
the exponent $\kappa$ must satisfy
$\kappa\ge m/(a-md)$. This naturally leads to the following 

\begin{question}\label{q:q}
Let $a=\liminf_{T\to\infty}\frac{\log |\{x\in {\sf X}(\mathbb{Z}):\,\, \|x\|<T\}|}{\log T}$, let 
$F:{\sf X}\to \mathbb{A}^m$ be a polynomial map of degree $d<a/m$,
and $\kappa> m/(a-md)$.
Given $\xi\in \R^m$, 
in which cases does the system \eqref{eq:main_ineq}
have solutions for all sufficiently small $\varepsilon$?
\end{question}	

An important source of motivation for the results we formulate below is the classical case of linear maps $F:\mathbb{R}^n \to \mathbb{R}^m$.  In this case, for generic (namely almost all) maps $F$ it is in fact possible to obtain the optimal rate. 
Indeed, it follows from the theory of inhomogeneous Diophantine approximation
that when $\kappa>m/(n-m)$, for all $\xi\in\mathbb{R}^m$ and almost all linear maps $F:\mathbb{R}^n \to \mathbb{R}^m$,
the inequality
$$
\|F(x)-\xi\|<\|x\|^{-1/\kappa}
$$
has infinitely many solutions $x\in\mathbb{Z}^n$.
This result matches the a-priori lower bound provided by the Pigeonhole Heuristics.

It is a natural but much more difficult challenge to produce non-linear examples
where the approximation property holds for all exponents $\kappa> m/(a-md)$, and such a result 
can be expected to give the {\it optimal density rate} for the distribution of the polynomial values in question. 
Our paper provides a series of examples exhibiting this property, including :
\begin{itemize}
\item quadratic forms,
\item linear maps on quadratic surfaces,
\item characteristic polynomial maps,
\item Gram-matrix maps.
\end{itemize}
In all these instances, we show that for generic maps of this form,
Question \ref{q:q} is answered positively.
A common property of these examples, generalizing the case of linear and quadratic maps mentioned already, is that the variety ${\sf X}$ is invariant under
an action of a linear algebraic group ${\sf G}\subset \hbox{GL}_n$.
This allows us to consider a family of maps 
$$
F_g:=F\circ g^{-1}:{\sf X}\to \mathbb{A}^m,\quad g\in {\sf G}(\mathbb{R}).
$$
We shall prove that given an arbitrary $\xi\in \mathbb{R}^m$,
for almost all $g\in {\sf G}(\mathbb{R})$ the values $F_g(x)$, $x\in {\sf X}(\Z)$, 
give a solution to \eqref{eq:main_ineq}
provided that $\kappa> m/(a-md)$ and $\varepsilon$ is sufficiently small.

Let us note that while the lower bound on $\kappa$ in Question \ref{q:q} is expressed in terms of geometric
and arithmetic data (namely dimension, degree, and the growth rate of the number of integral points), our proof relies on analytic estimates of certain averaging operators which ultimately depends on spectral bounds for automorphic representations, and also on the asymptotics of volume growth for balls in algebraic groups. 
It is a remarkable fact that all these parameters match 
precisely and give exactly the same exponent as predicted by the Pigeonhole Heuristics.

\subsection{Quadratic forms}\label{sec:quad0}
Let $Q(x)=\sum_{i,j=1}^n a_{ij}x_ix_j$ be a nondegenerate indefinite quadratic form in $n$ variables.
The density of values in this case has been extensively studied by both
number-theoretic and ergodic-theoretic methods. 
It was conjectured by Oppenheim in the 1920's that if $n\ge 5$ and 
$Q$ is irrational (that is, not proportional to a form with rational coefficients),
then the set of values $Q(x)$, $x\in \mathbb{Z}^n$, is dense in $\R$.
This problem has a long history that we will not attempt to discuss here in detail
and we refer instead to \cite{mar_s} for a comprehensive survey. 
Originally, it was approached by Fourier-analytic techniques
such as the Hardy--Littlewood Circle Method and its variants.
For instance, Davenport and Heilbronn \cite{DH} proved the Oppenheim conjecture
for non-degenerate indefinite diagonal quadratic forms $Q$ in $n\ge 5$ variables,
and Birch and Davenport \cite{BD} showed that in this setting for
$\kappa>2$, the system of inequalities
\begin{equation}
\label{eq:qq}
|Q(x)|<\varepsilon,\quad \|x\|<\varepsilon^{-\kappa}\quad \hbox{ with $x\in \mathbb{Z}^n\backslash \{0\}$}
\end{equation}
is solvable for all sufficiently small $\varepsilon>0$. This still seems to be the best known bound for diagonal forms, while based on the Pigeonhole Heuristics one might hope to have existence of solutions when $\kappa>1/(n-2)$ for non-degenerate forms $Q$ in $n$ variables.
In full generality, the Oppenheim conjecture was proved by Margulis \cite{M2,M1}, where it is 
shown that $Q(\mathbb{Z}^n)$ is dense in $\mathbb{R}$ for all non-degenerate indefinite irrational quadratic forms in $n\ge 3$ variables.
Margulis'  original proof was motivated by Ragunathan's observation that the Oppenheim conjecture would follow from a result regarding orbit closures of certain subgroups of the orthogonal group of $Q$
acting on the space of unimodular lattices.
This approach is topological in nature 
so that it was not clear originally how to derive any quantitative density estimates.
Subsequently, more refined analytic techniques for addressing this problem
have been developed by Bentkus and G\"otze \cite{BG} and G\"otze and Margulis \cite{GM}.
In particular, it was established in \cite{GM} that for any $\kappa>12$,
the system \eqref{eq:qq} has a solution for general non-degenerate indefinite irrational
quadratic forms in $n\ge 5$ variables. Moreover, 
it was shown in \cite{GM} that if $Q$ additionally satisfies
an explicit Diophantine condition, then for $\kappa>\kappa_0(Q)>1$,
the system
\begin{equation}
\label{eq:qq2}
|Q(x)-\xi|<\varepsilon,\quad \|x\|<\varepsilon^{-\kappa}\quad \hbox{ with $x\in \mathbb{Z}^n\backslash \{0\}$}
\end{equation}
is solvable for all $\xi\in \mathbb{R}$ and all sufficiently small $\varepsilon>0$.
It seems that the condition $n\ge 5$ is a natural barrier for applicability of the methods
of \cite{BG,GM}. The case of ternary quadratic forms $Q$ was investigated by Lindenstrauss and Margulis \cite{LM}.
They showed, in particular, that there exists $\kappa>0$ such that
for a non-degenerate indefinite ternary quadratic form $Q$ satisfying an explicit Diophantine condition, the system
$$
|Q(x)-\xi|<\varepsilon,\quad \|x\|<\exp(\varepsilon^{-\kappa})\quad \hbox{ with $x\in \mathbb{Z}^3\backslash \{0\}$}
$$
is solvable for all $\xi\in\mathbb{R}$ and all sufficiently small $\varepsilon>0$.
It is not known at present whether it is possible to find a solution of 
$|Q(x)-\xi|<\varepsilon$
with size $O(\varepsilon^{-\kappa})$ for some $\kappa>0$, as in the case of quadratic forms in $n\ge 5$ variables. 

Let us first consider the case of generic forms in three variable, where we establish the best-possible result for the rate of growth of solutions. 
In general, let us denote by $\mathcal{Q}({p,q};\ell)$ the set of real 
non-degenerate quadratic forms of signature $(p,q)$
with discriminant $\ell$.

\begin{Theorem}\label{ternary} 
Let $\ell\in \mathbb{R}\backslash\{0\}$, $\xi\in \mathbb{R}$, and $\kappa>1$.
Then for almost all $Q\in \mathcal{Q}({2,1};\ell)$, the system 
$$
|Q(x)-\xi|<\varepsilon,\quad \|x\|<\varepsilon^{-\kappa}\quad \hbox{ with $x\in \mathbb{Z}^3\backslash \{0\}$}
$$
has a solution for all $\varepsilon\in (0,\varepsilon_0(Q,\xi,\kappa))$.
\end{Theorem}

We note that in this setting $m=1$, $a=3$ and $d=2$, so that the Pigeonhole Heuristics (cf. Question \ref{q:q}) predicts density when $\kappa>1$, and Theorem \ref{ternary} establishes density when the rate $\kappa$ is in this range.
Our method can be applied to quadratic forms in any number of variables,
and we obtain solvability of \eqref{eq:qq2} with an explicit $\kappa>0$,
but this exponent is not optimal in general (cf. Theorem \ref{mainthm}--\ref{mainthm2} below).
After a conversation in June 2014 about a preliminary version of this paper in which Theorem \ref{ternary} was proved, Athreya and Margulis \cite{AM} subsequently obtained the same result for generic quadratic forms in three variables using the random Minkowski theorem \cite{mar_m}. Their method 
is expected to eventually give optimal $\kappa$ in all dimensions, but it does not seems to apply, for example, to the problem of establishing best-possible density estimates for linear forms on quadratic varieties, which we will consider next.

Theorem \ref{ternary} was first proved in 2014. 
Subsequently, a  paper \cite{GK1} by the first author and Kelmer also 
dealt with this question using the method of the present paper. 
Bourgain \cite{Bourgain} obtained a uniform generic result for certain diagonal forms using analytic techniques, and in \cite{GK2} analogues of these results were obtained for almost every quadratic form in three variables.

\subsection{Linear maps on quadratic surfaces}\label{sec:lin}
We now turn to investigate the quantitative density of linear maps defined on quadratic surfaces.
Let $Q$ be a non-degenerate indefinite rational quadratic form in $n$ variables, $k\in \Q$,
and let 
$$
{\sf X}=\{Q=k\}
$$  be the corresponding quadratic surface.
We fix a connected component $X^0$ of the surface ${\sf X}(\R)\backslash\{0\}$
that contains at least one  integral point.
Let $F:\mathbb{R}^n\to \R^m$ be a rational linear map of full rank.
We consider the set of values $F(x)$, $x\in X^0\cap \mathbb{Z}^n$.
In analogy with the Oppenheim conjecture, one expects that 
this set is dense in $\R^m$ provided that $F$ satisfies 
some basic geometric and Diophantine assumptions.
Results in this direction were proved by Sargent in \cite{S1}, but 
while the argument of \cite{S1} proves density,
it does not allow us to deduce any quantitative bounds as in \eqref{eq:main_ineq}.
Here we establish such a bound for generic linear maps.
Let ${\sf G}$ denote the special orthogonal group of the quadratic form $Q$
and $G={\sf G}(\R)^0$.
We will consider a family of linear maps $F_g(x)=F(g^{-1}x)$ for $g\in G$, and establish the following density  estimate for their generic values, under suitable conditions, as follows.

\begin{Theorem}\label{systems}
Assume that $n\ge 4$, $n=m+3$, $Q$ has signature $(n-1,1)$, and the form $Q \rvert_{F=0}$ is non-degenerate and indefinite. Let $\kappa>m$.
Then given any $\xi\in \mathbb{R}^m$, for almost all $g\in G$, the system
\begin{equation}
\label{eq:ql}
\|F_g(x)-\xi\|<\varepsilon,\quad \|x\|<\varepsilon^{-\kappa}\quad \hbox{ with $x\in X^0\cap\mathbb{Z}^n$}, 
\end{equation}	
has solutions for all 	$\varepsilon\in (0,\varepsilon_0(g,\xi,\kappa))$.
\end{Theorem}

Comparing this result with the Pigeonhole Heuristics (cf. Question \ref{q:q}),
we note that in this case $a=n-2$ and $d=1$, so that
Theorem \ref{systems} establishes the 
best possible quantitative density.
Our method can be used to show that when 
$n-m>3$, the system \eqref{eq:ql} is also solvable for some $\kappa>0$,
but this exponent is not optimal in general (cf. Theorem~\ref{mainthm} below). 

We also construct an example of a family of linear maps whose values fail to satisfy the Pigeonhole Heuristics.
Let $n\ge 4$ and 
$$
{\sf X}=\{x^{2}_1 + \dots + x^{2}_{n-1} - x^{2}_n=1\}
$$
be a quadratic surface.
For $1\le s\le n-1$, we consider a family of linear maps 
\begin{align*}
F_{{\alpha}}(x) &=x_n - \sum_{i = 1}^{s}\alpha_i x_i,
\quad\hbox{${\alpha}=(\alpha_1,\ldots,\alpha_s)\in \mathbb{R}^s$.}
\end{align*}
It is easy to check that for  $\|\alpha\|>1$, the restriction $Q|_{F_\alpha=0}$
is a non-degenerate indefinite form. Hence, it follows from \cite{S1}
that for all such irrational $\alpha\in \mathbb{R}^s$, the set of values
$\{F_\alpha(x):\, x\in {\sf X}(\mathbb{Z})\}$
is dense in the real line.
For a given $\xi\in \mathbb{R}$,
we investigate existence of integral solutions of the system
\begin{equation}
\label{eq:quad_contr}
|F_{{\alpha}}(x)-\xi|<\varepsilon,\quad \|x\|<\varepsilon^{-\kappa}\quad\hbox{ with $x\in {\sf X}(\mathbb{Z})$.}
\end{equation}
as $\varepsilon\to 0$.
A restriction on the exponent $\kappa$
can be deduced from	the Pigeonhole Heuristics. 
We observe that the set $V_\varepsilon$ of values
$F_{{\alpha}}(x)$ with $x\in \mathbb{Z}^n$ satisfying
$\|x\|\ll \varepsilon^{-\kappa}$ lies in an interval of length
$O(\varepsilon^{-\kappa})$ and has cardinality  $O(\varepsilon^{-\kappa(s+1)})$. Another bound on this cardinality
can be obtained by using that the number of $x\in {\sf X}(\mathbb{Z})$
satisfying $\|x\|<T$ is $O(T^{n-2})$.
This gives the bound on the cardinality 
$O(\varepsilon^{-\kappa(n-2)})$ which is better than the previous 
bound when $s>n-3$. One expects that the cardinality of the intersection
of $V_\varepsilon$ with a fixed interval is about $O(\varepsilon^{\kappa}|V_\varepsilon|)$.
Hence, if the approximation problem has a solution for all $\xi$
in a bounded interval, then 
we expect that $\varepsilon^{\kappa}|V_\varepsilon|\gg \varepsilon^{-1}$, and
$$
\min (\varepsilon^{-\kappa(s+1)},\varepsilon^{-\kappa(n-2)})\gg \varepsilon^{-1-\kappa}.
$$
We conclude that one might expect existence of solutions \eqref{eq:quad_contr}  for all sufficiently small $\varepsilon$ only when
$$
\kappa\ge  
\begin{cases}
\frac{1}{s},\quad s\le n-3,\\
\frac{1}{n-3},\quad s>n-3.
\end{cases}
$$
We shall show that when $s\le n-3$, the values of the linear forms $F_\alpha$
with generic $\alpha$  do not exhibit the quantitative density predicted by the Pigeonhole
Heuristics.

\begin{Prop}\label{th:naive}
	Let $\kappa_s< 1/(s-1)$ when $s\ge 2$ and $\kappa_1< 2$, and $\xi\in \mathbb{R}\backslash \mathbb{Z}$.
	Then for almost all ${\alpha}=(\alpha_1,\ldots,\alpha_s)\in \mathbb{R}^s$ and sufficiently small $\varepsilon>0$,
	the system	
	\begin{equation} \label{eq:c1}
	|F_{{\alpha}}(x)-\xi|<\varepsilon,\quad \|x\|<\varepsilon^{-\kappa_s}\quad\hbox{with $x\in {\sf X}(\mathbb{Z})$}
	\end{equation}
	has no solutions.	
\end{Prop}

For example, when $n=4$, according to Theorem \ref{systems}, the values of 
generic linear forms on ${\sf X}(\mathbb{Z})$ are quantitatively dense
with any $\kappa>1$ (as predicted by the Pigeonhole Heuristics),
but according to Proposition \ref{th:naive} there exists a two-dimensional
family of linear forms (given by scalar multiples of $F_\alpha$'s) whose values on 
${\sf X}(\mathbb{Z})$ are not quantitatively dense with any $\kappa< 2$.

\subsection{Characteristic polynomial map}\label{sec:char}
For a matrix $x\in \hbox{M}_3(\C)$, we consider its characteristic polynomial
$$
\det(t I-x)=t^3-F_2(x) t^2-F_1(x)t-F_0(x).
$$
Here $F_2(x)$ is the trace of the matrix $x$, $F_1$ is the sum of the diagonal minors of the matrix $x$,
and $F_0(x)$ is the determinant of the matrix $x$.
These polynomials play an important role in classical Invariant Theory, as they 
generate the algebra of conjugation invariant polynomials.
Let $\ell\in \Z\backslash \{0\}$ and 
$$
{\sf X}=\{x\in \hbox{M}_3(\C): \det(x)=\ell\}
$$ 
be the constant-determinant variety. We consider the polynomal map $F=(F_1,F_2):{\sf X}\to \mathbb{C}^2$.
We observe that the variety ${\sf X}$ is invariant under the action of 
the group ${\sf G}=\hbox{SL}_3\times \hbox{SL}_3$ defined by
$x\mapsto g_1 xg_2^{-1}$ for $g=(g_1,g_2)\in {\sf G}$.
Therefore, we also have a family of polynomial maps 
$F_g(x)=F(g_1^{-1}xg_2)$ for $g=(g_1,g_2)\in {\sf G}(\R)$. We can now state the following best-possible density estimate on their generic values. 

\begin{Theorem}\label{detmap}
	Let $\kappa>1$. Then given any $\xi\in \R^2$, for almost all $g=(g_1,g_2)\in {\sf G}(\R)$,
	the system 
	$$
	\|F_g( x)-\xi\|<\varepsilon,\quad \|x\|<\varepsilon^{-\kappa}\quad \hbox{ with $x\in {\sf X}(\mathbb{Z})$}, 
	$$	
	has solutions for all 	$\varepsilon\in (0,\varepsilon_0(g,\xi,\kappa))$.
\end{Theorem}

We remark that in this case $m=2$, $a=6$ and $d=2$,
so that the Pigeonhole Heuristics also gives the range of exponents $\kappa>1$ as in the theorem.

\subsection{Gram-matrix map}\label{sec:gram}

Let $Q$ be a non-degenerate indefinite quadratic form 
in $n\ge 3$ variables of signature $(p,q)$. 
To simplify the exposition, we additionally assume that $p\ge q$ and $\det(Q)=1$. 
We will also denote by $Q$ the corresponding 
bilinear form on $\mathbb{R}^n$. Given vectors $v_1,\ldots,v_n\in \R^n$,
the \emph{Gram matrix} is defined by
$$
F_Q(v_1,\ldots,v_n)=\left(Q(v_i,v_j)\right)_{i,j=1,\ldots,n}.
$$
A tuple of vectors $(v_1,\ldots, v_n)$ is called a unimodular frame 
if the vectors $v_i$ are linearly independent, and the lattice generated by them has covolume one. The collection of unimodular frames defines a hyper-surface ${\sf X}$
in $\hbox{M}_n(\C)$, and  ${\sf X}(\Z)$ consists of unimodular frames 
with integral coordinates. We note that $F_Q$ defines a surjective map
from ${\sf X}(\R)$ to $\hbox{Sym}(p,q;1)$, the set of real symmetric matrices
with signature $(p,q)$ and determinant one.
The distribution of the values $F_Q(x)$, $x\in {\sf X}(\Z)$, was studied in \cite{GW}. 
In particular, it follows from \cite[Corollary~1.3]{GW} that if 
$Q$ is irrational, then this set of values is dense in 
$\hbox{Sym}(p,q;1)$. Here we consider the problem of quantitative density 
in this setting. 
Since the points $F_Q(x)$, $x\in {\sf X}(\Z)$,
are dense in a proper subvariety of $\hbox{M}_n(\R)$,
the Pigeonhole Heuristics proposed above has to be modified.
We recall that it was shown in \cite{GW}
that the number of $(v_1,\ldots,v_n)\in {\sf X}(\Z)$ with 
$\|v_1\|,\ldots,\|v_n\|<T$ such that $F_Q(v_1,\ldots,v_n)$
is contained in a fixed bounded domain in 
$\hbox{Sym}(p,q;1)$ is bounded by $O(T^{(p-1)q})$ when $p>q$, and by $O(T^{(p-1)p}\log T)$ when $p=q$.
Since $\hbox{Sym}(p,q;1)$ has dimension $(n-1)(n+2)/2$,
we expect that this set of values can be dense at most on the scale
$O(T^{-\delta})$ with $\delta<\frac{2(p-1)q}{(n-1)(n+2)}$.
Equivalently, density on this scale corresponds to solvability of the system of inequalities
\begin{equation}
\label{eq:fq}
\|F_Q(x)-\xi\|<\varepsilon, \quad \|x\|<\varepsilon^{-\kappa}\quad
\hbox{ with $x\in {\sf X}(\Z)$.}
\end{equation}
with $\kappa>\frac{(n-1)(n+2)}{2(p-1)q}$.
Our next result states that this best-possible rate can be achieved for generic Gram matrix maps in dimension three.

\begin{Theorem}\label{th:gramm}
Let $\kappa>5$ and $\xi\in \hbox{\rm Sym}(2,1;1)$.
Then for almost all $Q\in \mathcal{Q}(2,1;1)$, 
the system \eqref{eq:fq} has a solution for all $\varepsilon\in (0,\varepsilon_0(Q,\xi,\kappa))$.
\end{Theorem}

The results stated so far are distinguished by the fact that they establish a density rate for polynomial values which matches the a-priori bound given by the Pigeonhole Heuristics. We now note that these results all fit into a more abstract setting, consisting of families
of polynomial maps parametrized by actions of algebraic groups on homogeneous varieties, and we turn to formulate systematic general results in this direction.  
We establish two results regarding the rate of density for values of polynomials maps, the first for homogeneous affine varieties, and the second for homogeneous projective varieties. We emphasize that the rates we establish are given explicitly by data associated with the corresponding algebraic groups,  the variety and the polynomial. This dependence will be explicated in the proofs of the results stated below, together with a criterion for when the rate established matches the bound given by the Pigeonhole Heuristics in the affine case. 

\subsection{General homogeneous varieties} 
Let ${\sf X}\subset \mathbb{A}^n$ be an affine algebraic variety 
defined over $\Q$ and equipped with 
a transitive action of a connected semisimple algebraic $\Q$-group ${\sf G}\subset \hbox{GL}_n$.
Let $X^0$ be a connected component of ${\sf X}$ such that 
$X^0\cap \Z^n\ne \emptyset$.
We note that $G:={\sf G}(\R)^0$ acts transitively on the connected components of ${\sf X}(\R)$.
Given a polynomial map $F:{\sf X}\to \mathbb{A}^m$ defined over $\R$,
we set $F_g=F\circ g^{-1}$ with $g\in G$.
We assume that $F$ is invariant under a connected semisimple algebraic
$\Q$-subgroup ${\sf H}$ of ${\sf G}$ which is isotropic over $\R$.
We assume that ${\sf H}$ is $\Q$-simple, and $H={\sf H}(\R)^0$ is totally non-compact in $G$, that is, for all the 
projections $\pi_i:G\to G_i$ to its simple factors, the closures of $\pi_i(H)$ are not compact.
(This, in particular, implies that $G$ has no compact factors.)

In this setting, we prove:

\begin{Theorem}\label{mainthm}
There exists $\kappa>0$ such that given $\xi\in \mathbb{R}^m$ which is a regular value of $F:X^0\to \R^m$, for almost all $g\in G$ the system
$$
\|F_g(x)-\xi\|<\varepsilon,\quad \|x\|<\varepsilon^{-\kappa}\quad \hbox{with $x\in X^0\cap \Z^n$}
$$
has solutions for all $\varepsilon\in (0,\varepsilon_0(g,\xi))$.
\end{Theorem}

Our second result concerns polynomial maps on projective varieties.
Let ${\sf G}\subset \hbox{GL}_n$ as before be a connected semisimple algebraic $\Q$-group.
We fix $x_0\in \Z^n$ such that the stabilizer of the line $[x_0]$ in the projective space $\mathbb{P}^{n-1}$
is a parabolic $\Q$-subgroup  of ${\sf G}$. Then ${\sf G}[x_0]\subset \mathbb{P}^{n-1}$ is a homogeneous projective variety of ${\sf G}$.
We consider the corresponding cone $X^0=G x_0$ in $\R^n$.
Let $F:\C^n\to \C^m$ be a polynomial map defined over $\R$
whose components are non-constant homogeneous polynomials. 
We assume that
$F$ is invariant under a connected semisimple algebraic
$\Q$-subgroup ${\sf H}$ of ${\sf G}$ which is isotropic over $\Q$.
As before, we also assume that ${\sf H}$ is $\Q$-simple, and $H={\sf H}(\R)^0$ is totally non-compact in $G={\sf G}(\R)^0$.

With these notations, we prove:

\begin{Theorem}\label{mainthm2}
	There exists $\kappa>0$ such that  for almost all $g\in G$ the system
	$$
	\|F_g(x)\|<\varepsilon,\quad \|x\|<\varepsilon^{-\kappa}\quad \hbox{with $x\in X^0\cap \Z^n$}
	$$
	has solutions for all $\varepsilon\in (0,\varepsilon_0(g))$.
\end{Theorem}

We note that the proofs of Theorems \ref{mainthm} and \ref{mainthm2} provide explicit bounds on the exponent $\kappa$ (see the formulas \eqref{eq:kkk} and \eqref{eq:kk2} in the proofs). However, in general these bounds are weaker than the bounds predicted by the Pigeonhole Heuristics.

\subsection{Outline of the proof}
We briefly sketch the strategy of the proofs of the above results.
We shall show that a proof of existence of integral solutions of a given size can be reduced to
an investigation of certain shrinking target problems on homogeneous spaces of algebraic groups.
Thus, fix an algebraic variety ${\sf X}\subset \mathbb{A}^n$ which invariant under an action 
of an algebraic group ${\sf G}$ and a polynomial 
map $F$ on ${\sf X}$. Consider a family of maps $x\mapsto F(g^{-1}x)$
parametrised by $g\in {\sf G}(\mathbb{R})$. We would like to establish 
a quantitative density of values $F(g^{-1}x)$, $x\in {\sf X}(\mathbb{Z})$,
which amounts to finding a solution to the inequality
\begin{equation}
\label{eq:uneq1}
\|F(g^{-1}x)-\xi\|<\varepsilon\quad\hbox{ with $x\in {\sf X}(\mathbb{Z})$,}
\end{equation}
with an explicit bound on the size of $x$.
We observe that the group $\Gamma:={\sf G}(\mathbb{Z})$ acts on the set
${\sf X}(\mathbb{Z})$, and we look for solutions of this inequality of the form $x=\gamma x_0$
for a fixed point $x_0\in {\sf X}(\mathbb{Z})$ and $\gamma\in \Gamma$.
Now if we further assume that the map $F$ is invariant under a subgroup $H$
of $G:={\sf G}(\mathbb{R})$, then one can show that the original problem can be reformulated
in terms of dynamics of the action of $\Gamma$ on the space $Y:=G/H$.
Namely, a solution of \eqref{eq:uneq1} can be constructed once we establish 
existence of
\begin{equation}
\label{eq:uneq2}
\gamma^{-1}gH\in T_\varepsilon\quad\hbox{ with $\gamma\in \Gamma,$}
\end{equation}
where $T_\varepsilon$ are certain shrinking subsets of $Y$.
A key ingredient in the solution of the latter problem is the duality principle in homogeneous dynamics, namely the reformulation of \eqref{eq:uneq2} in terms of 
dynamics of the action of $H$ on the space $Z:=\Gamma\backslash G$.
This amounts to analysis of another shrinking target problem:
\begin{equation}
\label{eq:uneq3}
\Gamma gh\in S_\varepsilon\quad\hbox{ with  $h\in H,$}
\end{equation}
where $S_\varepsilon$ are suitable shrinking subsets of $Z$.
To establish solvability of \eqref{eq:uneq3} with an explicit bound on $h\in H$,
we investigate the behavior of the following averaging operators
$$
\pi_Z(\beta_t)f(z)=\frac{1}{m_H(H_t)}\int_{H_t}f(zh)dm_H(h),\quad z\in Z,
$$
defined for functions $f$ on $Z$ and subsets $H_t$ of $H$, where $m_H$ denotes
a Haar measure on $H$.
The method that we use constitutes a variation of the method employed in \cite{GGN3} to obtain upper bounds for rates of distribution of dense lattice orbits.

\subsection{Structure of the paper.}
In the next section, we will investigate a general shrinking target problem
for group actions and relate this problem to estimates on certain averaging operators.
In Section \ref{sec:unitary}, we review properties of unitary representations
of semisimple groups which will allow us to deduce quantitative mean ergodic theorem
for the averaging operators introduced in Section \ref{sec:g_gen}.
Then in Section \ref{sec:quad}, we discuss automorphic representation of orthogonal groups of indefinite quadratic forms 
and establish temperedness of restrictions of these representations.
The general Theorems \ref{mainthm}--\ref{mainthm2} are proved in Section \ref{sec:proof1},
and Theorems \ref{ternary}--\ref{th:gramm} are proved in Section \ref{sec:proof2}.
Finally, in Section \ref{sec:naive}, we give a proof of Proposition \ref{th:naive}, which
is elementary and independent of the rest of the paper.

\section{General shrinking target problems}
\label{sec:g_gen}

In the present section we will establish a solution to a general shrinking target problem
that will appear in the proofs of our main results.
Let $H$ be a (non-compact) locally compact second countable unimodular group
acting (on the right) on 
a standard probability space $(Z,\mu_Z)$, preserving the measure $\mu_Z$. 
Fix a family $\mathcal{S}=\{S_\varepsilon\}_{\varepsilon\in (0,\varepsilon_0)}$ 
of shrinking measurable subsets of $Z$. 
We would like to estimate the rate at which the $H$-orbits in $Z$ visit the target subsets $S_\varepsilon$.
To make this question precise, we fix a proper measurable function
$|\cdot |:H\to \R^+$ on $H$, and view it as a gauge that measures the size of elements in $H$.
Then given $\kappa>0$ and $z\in Z$, we seek to find solutions for
\begin{equation}
\label{eq:shrink0}
z h \in S_\varepsilon,\quad |h|<\varepsilon^{-\kappa}\quad\hbox{ with $h\in H$. }
\end{equation}
Our first result shows that \eqref{eq:shrink0} can be analyzed using suitable
averaging operators. We set
$$
H_t=\{h\in H:\, |h|<t\}
$$
and consider a family of linear operators
$\pi_Z(\beta_t) :L^2(Z,\mu_Z)\to L^2(Z,\mu_Z)$ defined by
$$
\pi_Z(\beta_t)f(z)=\frac{1}{m_H(H_t)}\int_{H_t}f(zh)dm_H(h),\quad z\in Z,
$$
where $m_H$ denotes a fixed choice of a Haar measure on $H$.

We make the following assumptions:
\begin{enumerate}
\item[A1.] (\emph{volume lower bound for the family $H_t$}). The $\liminf$ 
$$
b:=\liminf_{t\to \infty} \frac{\log m_H(H_t)}{\log t}
$$
is positive.

\item[A2.] (\emph{effective mean ergodic theorem}).  
There exists $\theta > 0$ such that for every $\eta > 0$ and  $t \ge t_\eta$,
$$
\left\|\pi_Z(\beta_t)f-\int_Zf\, d\mu_Z\right\|_{L^2(Z,\mu_Z)}\ll_\eta m_H(H_t)^{-\theta +\eta}\,\|f\|_{L^2(Z,\mu_Z)}
$$
for all $f\in L^2(Z,\mu_Z)$.

\item[A3.] (\emph{local dimension bound for targets}). The $\limsup$
$$
\zeta:= \limsup_{\varepsilon\to 0^+}\frac{\log(\mu_{Z}({S}_\varepsilon))}{\log \varepsilon}
$$
is finite.
\end{enumerate}

Under the assumptions A1--A3, we prove:

\begin{Prop}\label{p:upper}
	Let $\kappa >  \frac{\zeta}{2\theta b}$. Then for almost all $z\in Z$
	and $\varepsilon\in (0,\varepsilon_0(z,\kappa))$, there exists $h\in H$ satisfying
	$$
	z h\in S_\varepsilon\quad\hbox{and}\quad |h|<\varepsilon^{-\kappa}.
	$$
\end{Prop}

\begin{proof}
Let 
$$
Z(t,\varepsilon)=\{z\in Z:\, zH_t \cap S_\varepsilon=\emptyset\},
$$
and let $f_\varepsilon$ denote the characteristic function of the set $S_\varepsilon$.
Then $\pi_Z(\beta_t)f_\varepsilon(z)=0$ for $z\in Z(t,\varepsilon)$.
Utilizing the effective mean ergodic theorem (Assumption A2), we have for each given $\eta > 0$ and all sufficiently large $t$,
$$
\norm{\pi_Z(\beta_t)f_\varepsilon-\int_Z f_\varepsilon\, d\mu_{Z}}_{L^2(Z,\mu_Z)}\ll_\eta m_H(H_t)^{-\theta +\eta}\, \norm{f_\varepsilon}_{L^2(Z,\mu_Z)}.
$$
By integrating only on the subset $Z(t,\varepsilon)$,  this estimate implies that
\begin{equation}
\label{eq:zzz}
\mu_{Z}(Z(t,\varepsilon))^{1/2}\left( \int_{Z} f_\varepsilon \,d\mu_{Z}\right)\ll_\eta  m_H(H_t)^{-\theta+\eta}\, \|f_{\varepsilon}\|_{L^2(Z,\mu_Z)}.
\end{equation}
Since 
$$
\int_{Z} f_\varepsilon \,d\mu_{Z}=\mu_Z(S_\varepsilon)\quad\hbox{ and }\quad
\|f_{\varepsilon}\|_{L^2(Z,\mu_Z)}=\mu_Z(S_\varepsilon)^{1/2},
$$
we obtain the bound
$$
\mu_Z(Z(t,\varepsilon))\ll_\eta  m_H(H_t)^{-2(\theta-\eta)}\mu_Z(S_\varepsilon)^{-1}
$$
for all sufficiently large $t$.
By Assumption A1, for all $\eta>0$ and all sufficiently large $t$, we have 
$$
m_H(H_t)\gg_\eta t^{b-\eta},
$$
and by Assumption A3, for all $\eta>0$ and all sufficiently small $\varepsilon>0$,
$$
\mu_Z(S_\varepsilon)\gg_\eta \varepsilon^{\zeta+\eta}.
$$
Therefore, putting these estimates together, 
\begin{align*}
	\mu_Z(Z(t,\varepsilon))\ll_{\eta }
	t^{-2(b-\eta)(\theta-\eta)} \varepsilon^{-(\zeta+\eta)}.
\end{align*}
Now we take $\varepsilon=t^{-1/\kappa}$ 
and set $Z(t):=Z(t,t^{-1/\kappa})$.
This gives the estimate
$$
\mu_Z(Z(t))\ll_{\eta} t^{(\zeta+\eta)/\kappa  -2(b-\eta)(\theta-\eta)}
$$
valid for all $\eta>0$, $t>t_\eta$, and $\varepsilon\in (0,\varepsilon_\eta)$.
Since $\kappa > \frac{\zeta}{2\theta  b }$, we can choose
$\eta$ sufficiently small to obtain that for some exponent $\sigma>0$
and all sufficiently large $t$,
$$
\mu_Z(Z(t))\ll t^{-\sigma}.
$$
In particular, we conclude that  
$$
\sum_{n=1}^\infty \mu_Z(Z(2^n)) < \infty,
$$
and by the Borel--Cantelli Lemma, almost every point $z\in Z$ eventually  avoids 
the sets $Z(2^n)$. Equivalently, for almost every $z\in Z$,
there exists $n_0(z)\in \N$ such that 
$$
zH_{2^n} \cap S_{2^{-n/\kappa}}\ne \emptyset \quad\hbox{ for all $n \ge n_0(z)$.}
$$ 
Hence, we have shown that 
for almost every $z\in Z$ and $n\ge n_0(z)$,
there exists $h\in H$ such that
\begin{equation}
\label{eq:1}
zh\in S_{2^{-n/\kappa}}\quad\hbox{and}\quad |h|< 2^n.
\end{equation}
Given $\varepsilon\in (0,1)$, choose $n$ such that $2^{-n/\kappa}\le\varepsilon<2^{-(n-1)/\kappa}$.
Then \eqref{eq:1} implies that for almost every $z\in Z$ and $\varepsilon\in (0, \varepsilon_0(z,\kappa))$,
$$
zh\in S_{\varepsilon}\quad\hbox{and}\quad |h|<  2\varepsilon^{-\kappa}.
$$
It also follows that for any $\kappa'>\kappa$ and 
sufficiently small $\varepsilon$,
$$
zh \in S_{\varepsilon}\quad\hbox{and}\quad |h|< \varepsilon^{-\kappa'}.
$$
This completes the proof.
\end{proof}	

Our next result concerns shrinking targets for lattice actions on homogeneous spaces.
Let $G$ be a (non-compact) locally compact second countable unimodular group, $H$ a closed unimodular subgroup of $G$, and  $\Gamma$ be a lattice subgroup of $G$. 
We consider the action of $\Gamma$ on the homogeneous space $Y:=G/H$
and investigate the following shrinking target problem.
Let $|\cdot |:G\to \R^+$ be a proper measurable function on $G$ and 
$\mathcal{T}=\{T_\varepsilon\}_{\varepsilon\in (0,\varepsilon_0)}$ a family of 
relatively compact shrinking measurable subsets of $Y$.
Given $\kappa>0$ and $y\in Y$, we seek to find solutions to
\begin{equation}
\label{eq:shrink}
\gamma^{-1}y\in T_\varepsilon,\quad |\gamma|<\varepsilon^{-\kappa}\quad\hbox{ with $\gamma\in\Gamma$. }
\end{equation}
Without loss of generality, we may assume that the subsets $T_\varepsilon$ are of the form
$$T_\varepsilon=\tilde S_\varepsilon H$$
for some relatively compact shrinking measurable subsets $\tilde S_\varepsilon$ of $G$.

We fix Haar measures $m_G$ on $G$. 
Then $Z=\Gamma\setminus G$ is equipped with a finite $G$-invariant measure $m_{Z}$ satisfying
$$
\int_G f(g)\,dm_G(g)=\int_{Z}\Big(\sum_{\gamma\in \Gamma}f(\gamma g)\Big) dm_{Z}(\Gamma g)\quad\hbox{ for $f\in C_c(G)$.}
$$
We note that  $m_{Z}$ is not necessarily a probability measure, and its total mass  is equal to the Haar measure  of a fundamental domain 
of $\Gamma$ in $G$ which we denote by $V(\Gamma)$.
Thus $m_{Z}(Z)=V(\Gamma)$. We denote by $\mu_Z$ the probability measure $m_{Z}/V(\Gamma)$. 

We shall make the following assumptions:

\begin{enumerate}

\item[A3$^\prime$.] (\emph{local dimension bound for targets}). The $
\limsup$
$$
\zeta:= \limsup_{\varepsilon\to 0^+}\frac{\log(m_{G}(\tilde {S}_\varepsilon))}{\log \varepsilon}
$$
is finite.

\item[A4.] (\emph{coarse admissibility of $|\cdot|$}). 
For every relatively compact set $\Omega\subset G$, there exists $c_1(\Omega) =c_1 > 0$ such that  
$$
|\omega_1 g \omega_2|\le c_1 |g|\quad \hbox{for all $\omega_1,\omega_2\in \Omega$ and $g\in G$.}
$$

\end{enumerate}

Under the assumptions A1, A2, A3$^\prime$ and  A4, we prove:

\begin{Prop}\label{p:upper2}
Let $\kappa > \frac{\zeta}{2\theta b}$. Then for almost all $g\in G$
and $\varepsilon\in (0,\varepsilon_0(g,\kappa))$, there exists $\gamma\in \Gamma$ 
satisfying
$$
\gamma^{-1}gH\in T_\varepsilon\quad\hbox{and}\quad |\gamma|<\varepsilon^{-\kappa}.
$$
\end{Prop}

\begin{proof}
Let $S_\varepsilon=\Gamma \widetilde{S}_\varepsilon$.
We note that the sets $\tilde{S}_\varepsilon$ are contained in a fixed relatively compact subset of $G$ that we denote by $\Omega$.
We observe that
since the intersection of the lattice $\Gamma$ with the relatively compact set $\Omega\cdot \Omega^{-1}$ is finite and has at most, say, $N(\Omega)$ elements, 
the map $g\mapsto \Gamma g$ is at most $N(\Omega)$-to-$1$ on $\Omega$, and it follows that 
$$
\mu_Z(S_\varepsilon)=V(\Gamma)^{-1} m_Z(S_\varepsilon)
\ge V(\Gamma)^{-1} N(\Omega)^{-1} m_G(\widetilde{S}_\varepsilon).
$$
In particular, we conclude that the sets $S_\varepsilon$ satisfy Condition A3.
We apply Proposition \ref{p:upper} to the action of $H$ on $Z$ and the targets $S_\varepsilon\subset Z$ to deduce that for almost all $g\in G$ and $\varepsilon\in (0,\varepsilon_0(g,\kappa))$, there exists $h\in H$ satisfying
$$
\Gamma gh \in S_\varepsilon\quad \hbox{and}\quad |h|<\varepsilon^{-\kappa}.
$$
It follows that there exists $\gamma\in\Gamma$ such that $\gamma^{-1}gh\in \widetilde{S}_\varepsilon$.
In particular this implies that
$$
\gamma^{-1}gH\in \widetilde{S}_\varepsilon H=T_\varepsilon.
$$
Since  $\gamma\in gh\widetilde{S}_\varepsilon^{-1}\subset gh\Omega^{-1}$,
it also follows from Condition~A4 that 
there exists $c=c(g,\Omega)>0$ such that
$|\gamma|\le c|h|$. 
Hence, we have shown that for almost all $g\in G$ and $\varepsilon\in (0,\varepsilon_0(g,\kappa))$, there exists $\gamma\in \Gamma$ satisfying
$$
\gamma^{-1}gH\in T_\varepsilon\quad\hbox{and}\quad |\gamma|<c\, \varepsilon^{-\kappa}.
$$
This also implies that for any $\kappa'>\kappa$ and 
sufficiently small $\varepsilon$, we have 
$|\gamma|< \varepsilon^{-\kappa'}.$
This proves the proposition.
\end{proof}

\section{Unitary representations and ergodic theorems} 
\label{sec:unitary}

A crucial ingredient of our arguments is a general mean  ergodic theorems for the actions
of semisimple Lie groups \cite{N,GN1,GN2} that we now recall.
Let $H$ be a connected semisimple Lie group with finite centre.
We fix a Haar measure $m_H$ on $H$.
We say that a unitary representation $\pi:H\to U(\mathcal{H})$ of $H$ is \emph{$L^p$-integrable}
if for all vectors $v,w$ in a dense subspace of $\mathcal{H}$, the matrix
coefficients $\left<\pi(h)v,w\right>$, $h\in H$, are in $L^p(H)$. 
We also say that the representation is \emph{$L^{p+}$-integrable} if 
it is $L^q$-integrable for all $q>p$.

We recall that a large class of unitary representations of $H$ arising from actions on
homogeneous spaces are known to be $L^p$-integrable. Let ${\sf G}$ be a connected semisimple
algebraic $\mathbb{Q}$-group. We set $G={\sf G}(\mathbb{R})^0$, $\Gamma={\sf G}(\mathbb{Z})\cap G$, and consider the space $Z=\Gamma\backslash G$ equipped with the probability Haar measure $\mu_Z$. Let ${\sf H}$ be a connected $\Q$-simple algebraic subgroup of ${\sf G}$ 
defined over $\Q$ and $H={\sf H}(\mathbb{R})^0$.
The following theorem is a consequence of property $\tau$, namely the uniform spectral gap property for 
automorphic representations, which is proved in full generality by Clozel \cite{Clozel},
combined with the functorial properties of the automorphic spectrum, established by Burger and Sarnak \cite{BS}. We refer to \cite{Sarnak} and \cite{Cl07} for an exposition of these results.

\begin{Theorem} \label{th:spectral gap}
There exists $p=p({\sf H})\ge 2$ such that 	
given any closed $H$-invariant subspace $\mathcal{H}$ of 
$L^2(Z,\mu_Z)$ that does not contain any non-zero $H$-invariant vectors,
the unitary representation of $H$ on $\mathcal{H}$ is $L^p$-integrable.
\end{Theorem}

To elucidate Theorem \ref{th:spectral gap}, note that according to \cite[Th.~1.1(a)]{BS}, the representation of $H$ on $\cH$ 
is weakly contained in the set of automorphic representations. 
Since ${\sf H}$ is $\Q$-simple, its simply connected cover $\tilde{\sf H}$ is obtained by restriction of scalars
from an absolutely simple simply connected algebraic group ${\sf M}$ defined over a number field $\mathbb{K}$.
Then 
$$
H={\sf H}(\R)^0\simeq \tilde{\sf H}(\R)= \prod_{v\in V_\infty} {\sf M}(\mathbb{K}_v),
$$
where the product is taken over all Archimedean completions of $\mathbb{K}$.
According to \cite[Th.~3.1]{Clozel}, the automorphic representations of simple factors
${\sf M}(\mathbb{K}_v)$ are isolated from the trivial representations.
This implies that the representation $H$ is $L^p$ integrable for some $p\ge 2$
where the integrability exponent is determined by bounds on the automorphic spectrum of ${\sf M}$.

The important additional feature of the uniformity of the integrability exponent will not play a role in the applications that we discuss here. We note however that this feature can be used to establish existence of integral solutions of \eqref{eq:main_ineq} satisfying additional congruence conditions, with a uniform bound $\kappa$ controlling their rate of growth.

It follows from the Moore Ergodicity Theorem \cite{Moore} that 
in many cases the subspace $\mathcal{H}$ in Theorem \ref{th:spectral gap} can be taken to be
$\mathcal{H}=L^2_0(Z,\mu_Z)$, namely, the space of $L^2$-integrable functions on $Z$ satisfying
$\int_Z f\, d\mu_Z=0$. We say that the group $H$ is totally non-compact in $G$ if for the 
projections $\pi_i:G\to G_i$ to its simple factors, the closures of $\pi_i(H)$ are not compact.
(This, in particular, implies that $G$ has no non-trivial compact factors.)
It follows from \cite{Moore} that 
if $H$ is totally non-compact, then given any unitary representation $\pi:G\to U(\mathcal{H})$,
the sets of $H$-invariant and $G$-invariant vectors coincide. 
Hence, we obtain the following corollary.

\begin{Cor} \label{cor:spectral gap}
	Under the foregoing assumptions on ${\sf H}$ and ${\sf G}$, suppose in addition that $H$ is totally non-compact in $G$.
	Then there exists $p=p({\sf H})\ge 2$ such that 
	the unitary representation of $H$ on $L_0^2(Z,\mu_Z)$ is $L^p$-integrable.
\end{Cor}

Let us consider a measure-preserving action of $H$ on a standard probability space $(Z,\mu_Z)$.
For a measurable subset $B\subset H$ with $0<m_H(B)<\infty$,
we define the averaging operator
$$
\pi_Z(\beta)f(z)=\frac{1}{m_H(B)}\int_{B}f(zh)dm_H(h),\quad z\in Z,
$$
for $f\in L^2(Z,\mu_Z)$. We recall the following mean ergodic theorem (see \cite[Th.~4.5]{GN2}):

\begin{Theorem}\label{th:mean}
	Let $\mathcal{H}$ be a closed $H$-invariant subspace of $L^2(Z,\mu_Z)$
	such that the corresponding unitary represenation of $H$ on $\mathcal{H}$ is $L^{p+}$-integrable. Then for all $f\in \mathcal{H}$,
	$$
	\left\|\pi_Z(\beta)f\right\|_{L^2(Z,\mu_Z)}\ll_\eta m_H(B)^{-\frac{1}{2n_e(p)}+\eta}\,\|f\|_{L^2(Z,\mu_Z)},\quad \eta>0,
	$$
	where $n_e(p)$ is defined as $n_e(p)=1$ when $p=2$ and as the least even integer $\ge p/2$ when $p>2$.
\end{Theorem}

An important role in our considerations is also played by the notion of weak containment of unitary representations. We say that a unitary representation $\pi_1$ of $H$ is \emph{weakly contained} 
in a unitary representation $\pi_2$ of $H$ if every diagonal matrix coefficient
$\left<\pi_1(h)v,v\right>$ for the representation $\pi_1$ can be approximated
uniformly on compact subsets of $H$ by finite sums of 
diagonal matrix coefficient
$\left<\pi_2(h)w_i,w_i\right>$ for the representation $\pi_2$.
We refer to \cite[App.~F]{bhv} for an extensive discussion of the weak containment property.
We denote by $\lambda_H$ the regular representation of $H$ on $L^2(H)$.
For a closed subgroup $H_0$ of $H$,
\begin{itemize}
\item if $\pi_1$ is weakly contained in $\pi_2$,
then $\pi_1|_{H_0}$ is weakly contained in $\pi_2|_{H_0}$,
\item $\lambda_H|_{H_0}$ is weakly contained
in the regular representation $\lambda_{H_0}$.
\end{itemize}

A unitary representation $\pi$ of $H$ is called \emph{tempered}
if it is weakly contained in the regular representation $\lambda_H$.
According to \cite{chh}, in our set-up a unitary representation
is tempered if and only if it is $L^{2+}$-integrable.

We list several examples of $L^p$-integrable and tempered unitary representations that will play 
a role in our arguments:

\begin{enumerate}
	\item[(i)] Let $G$ be a connected simple Lie group with finite centre of rank $\ge 2$.
	Then there exists a uniform $p\ge 2$ such that any unitary representation of
	$G$ without $G$-invariant vectors is $L^p$-integrable.
	In particular, for $G=\hbox{SL}_n(\mathbb{R})$, $n\ge 3$, 
	such representations are known to be $L^{2(n-1)+}$ integrable.
	
	\item[(ii)] Let $G=\hbox{SL}_3(\mathbb{R})$ and $H=\hbox{SO}_Q(\mathbb{R})^0\subset G$
	for a ternary non-degenerate indefinite quadratic form $Q$.
	Then for any unitary representation $\pi$ of $G$ without $G$-invariant vectors, the representation $\pi|_H$ is tempered.
	
	\item[(iii)] Let $G=\hbox{SL}_3(\mathbb{R})\times \hbox{SL}_3(\mathbb{R})$
	and $H\simeq \hbox{SL}_3(\mathbb{R})$ be the diagonal subgroup of $G$.
	Then for any unitary  
	representation $\pi$ of $G$ without invariant vectors
	for the $\hbox{SL}_3(\mathbb{R})$-factors of $G$, the representation $\pi|_H$ is tempered.
	
	\item[(iv)] 
	Let ${G}=\hbox{SO}_Q(\mathbb{R})^0$ where 
	$Q$ is a non-degenerate rational quadratic form in $n\ge 4$ variables of signature $(n-1,1)$.
	Let $H$ be a closed subgroup of $G$ such that 
	$H\simeq \hbox{SO}(2,1)^0$. We consider the unitary representation 
	$\pi$ of $G$ on $L^2_0(\Gamma\backslash G)$ for a congruence subgroup $\Gamma$ of $G$.
	Then the restriction $\pi|_H$ is tempered.
\end{enumerate}

(i) follows from \cite{li,O1,O2},
(ii) and (iii) are observed in \cite[Sec.~4]{GGN3}, and (iv) will be proved in the next section.

\section{Automorphic representations of orthogonal groups}
\label{sec:quad}

Let $Q$ be a non-degenerate rational quadratic form in $n$ variables of signature $(n-1,1)$, $n \ge 4$.
We denote by ${\sf G}=\hbox{SO}_Q\subset \hbox{GL}_n$
the special orthogonal group of $Q$ and set $G={\sf G}(\R)^0$.
Let $H$ be any closed subgroup of $G$ which satisfies 
$H\simeq \hbox{SO}(2,1)^0$. We fix a congruence subgroup $\Gamma$ of $G$. Our aim is to analyse the unitary representation  arising from the action of $H$ on $L_0^2(\Gamma\backslash G)$. In this section, we prove:

\begin{Prop} \label{p:spherical}
	For $n\ge 4$, the unitary representation of $H$ on $L_0^2(\Gamma\backslash G)$ is tempered.
\end{Prop}

\begin{proof} 
Let $\pi$ denote the unitary representation of $G$ on $L_0^2(\Gamma\backslash G)$.
We consider its direct integral decomposition 	
$\pi=\int^\oplus \pi_z$
into irreducible representations $\pi_z$ of $G$.
We shall use the properties of unitary representations reviewed in Section~\ref{sec:unitary}.
It is sufficient to show that the representation $\pi|_H$ is $L^{2+}$-integrable.
We note that $\pi|_H$ is $L^p$-integrable provided that 
$\pi_z|_H$ is $L^p$-integrable for almost all $z$. 
Hence, it remains to verify that for the representations $\pi_z$,
which appear in the direct integral decomposition of $\pi$, the restrictions  
$\pi_z|_H$ are $L^{2+}$-integrable.
When $\pi_z$ is a tempered representation of $G$, it is weakly contained in the regular representation $\lambda_G$.
Also $\pi_z|_H$ is weakly contained in $\lambda_G|_H$,
and $\lambda_G|_H$ is weakly contained in $\lambda_H$.
Hence, we conclude that in this case 
the restriction $\pi_z|_H$ is tempered and, in particular, $L^{2+}$-integrable. 
Now it remains to
analyse the restrictions $\pi_z|_H$ for non-tempered representations $\pi_z$.

There exists a Cartan involution of $G$ which defines a Cartan involution of $H$
(see, for instance, \cite[p. 53]{most}).
Let $K$ be the corresponding maximal compact subgroup of $G$. Then $K_H=K\cap H$
is a maximal compact subgroup of $H$.
We choose a Cartan subgroup $A$ of $H$.
Since $G$ has rank one, $A$ is also a Cartan subgroup of $G$.
We have Cartan decompositions 
$$
G=KA^+K\quad\hbox{and}\quad H=K_HA^+K_H.
$$
Let $\mathfrak{a}=\hbox{Lie}(A)$.
The adjoint action of $\mathfrak{a}$ on the Lie algebra of $G$
has exactly one simple root $\alpha$ of multiplicity $n-2$.
We fix $X_0\in \mathfrak{a}$ such that $\alpha(X_0)=1$.
This gives the identification of $\mathfrak{a}\simeq \mathbb{R}$
which follows the standard convention that the half-sum of positive roots
corresponds to $\rho_{n-1}:=\frac{n-2}{2}$.
With this normalisation the Haar measure on $H$ is given by
\begin{equation}
\label{eq:haar}
\int_H f\, dm_H=\int_{K_H\times \R^+\times K_H} f(k_1\exp(tX_0)k_2)\, \sinh(t) dk_1dtdk_2,\quad f\in C_c(H).
\end{equation}

We now recall some facts from the classification of the irreducible unitary representations of $G\simeq \hbox{SO}(n-1,1)^0$ \cite{h,ks,T,BSB}. It is known that all non-tempered irreducible representations of $G$ are given by the complementary series representations $J(\omega,s)$
which are constructed as follows.
Let $P=MAN$ be a proper parabolic subgroup  of $G$.
For $s\in \mathbb{C}$, we define a character of $A$ by $\chi_s(\exp(tX_0))=e^{st}$.
Given  an irreducible finite-dimensional representation 
$\omega$ of $M$ and $s\in\mathbb{C}$ with $\hbox{Re}(s)>0$, we construct the unitary 
representation of $G$ which is induced from the representation $\sigma\otimes \chi_s$
of $P$. Then $J(\omega,s)$ is the unique irreducible quotient of this representation.
All non-tempered irreducible unitary represenations of $G\simeq \hbox{SO}(n-1,1)^0$ 
are of this form $J(\omega,s)$ with $s\in I(\sigma)$ for an interval $I(\sigma)\subset (0,\rho_{n-1}]$. When $\sigma=1_M$ is the trivial representation of $M$, the interval is 
$I(1_M)= (0,\rho_{n-1}]$. This gives the spherical complementary series representations of $G$.
When $\sigma\ne 1_M$, we always have $I(\sigma)\subset (0,\rho_{n-1}-1]$
(see, for instance, \cite[Prop.~49--50]{ks}).
The $K$-finite matrix coefficients $F$ of the representations $J(\omega,s)$
satisfy the bound
\begin{equation}
\label{eq:F}
|F(k_1\exp(tX_0)k_2)|\ll_{F,\varepsilon} e^{(s-\rho_{n-1}+\varepsilon)t}\quad\hbox{$k_1,k_2\in K$ and $t\in \mathbb{R}^+$}
\end{equation}
for all $\varepsilon>0$ (see \cite{CM} or \cite[p.~253]{k}).

Let us consider the case when $\pi_z$ 
appearing in the decomposition of $\pi$ 
is a non-tempered spherical irreducible representation of $G$. 
While \eqref{eq:F} gives allows for arbitrarily slow decay when $s\to\rho_{n-1}$,
it was shown in \cite[Thm. 1.2(a)]{BS} that when $n \ge 4$, 
the complementary series representations $J(1_M,s)$ which appear in
the decomposition of $L^2_0(\Gamma\backslash G)$ must have 
$s \le \rho_{n-1}-\frac{1}{2}$. Hence, we deduce that for all 
$K$-finite matrix coefficients $F$ of $\pi_z$, we have the bound
$$
|F(k_1\exp(tX_0)k_2)|\ll_{F,\varepsilon} e^{(-1/2+\varepsilon)t}\quad
\hbox{for $k_1,k_2\in K$ and $t\in \mathbb{R}^+$.}
$$
In view of \eqref{eq:haar}, it follows that $F|_H$
are in $L^{2+\eta}(H)$ for all $\eta>0$. Thus, $\pi_z|_H$ is 
$L^{2+}$-integrable, as required.

Now suppose that $\pi_z$ is non-tempered non-spherical irreducible representation of $G$.
Then $\pi_z$ can be identified with $J(\omega,s)$ with $s\le\rho_{n-1}-1$.
Hence, all 
$K$-finite matrix coefficients $F$ of $\pi_z$ satisfy
$$
|F(k_1\exp(tX_0)k_2)|\ll_{F,\varepsilon} e^{(-1+\varepsilon)t}\quad
\hbox{for $k_1,k_2\in K$ and $t\in \mathbb{R}^+$.}
$$
Then it follows that  $F|_H$
are in $L^{2+\eta}(H)$ for all $\eta>0$ (in fact, in $L^{1+\eta}(H)$ for $\eta>0$),
and $\pi_z|_H$ is 
$L^{2+}$-integrable.
This completes the proof that the representation $\pi|_H=\int^\oplus (\pi_z|_H)$ is tempered.
\end{proof}

\section{Proofs of Theorems \ref{mainthm}--\ref{mainthm2}}
\label{sec:proof1}

\begin{proof}[Proof of Theorem \ref{mainthm}]
	We recall that $X^0$ denotes a fixed connected component of ${\sf X}(\R)$ such that $X^0\cap \Z^n\ne \emptyset$. We fix $x_0\in X^0\cap \Z^n$. Since ${\sf X}(\R)$ consists of finitely many open (and closed) orbits of $G={\sf G}(\R)^0$ (see \cite[\S 3.2]{PR}), we have $X^0=G x_0$. 
	Moreover, we recall that the map $g\mapsto gx_0$, $g\in G$, is a submersion.
	We pick $g_0\in G$ such that 
	$F(g_0^{-1} x_0)=\xi$ and $F$ is a submersion at $g_0^{-1}x_0$.
	Then the map $G\to \R^m:g\mapsto F(g^{-1}x_0)$ is also a submersion at $g_0$.
	Using a canonical coordinate system for this submersion in a neighbourhood of $g_0$,
	one can construct a collection of compact shrinking subsets $\widetilde{S}_\varepsilon$,
	$\varepsilon\in (0,\varepsilon_0)$ of $G$ such that 
	\begin{equation}
	\label{eq:s}
	\|F(s^{-1}x_0)-\xi\|<\varepsilon\quad\hbox{for all $s\in \widetilde{S}_\varepsilon$},
	\end{equation}
	and 
	\begin{equation}
	\label{eq:e}
	m_G(\widetilde{S}_\varepsilon)\gg \varepsilon^m\quad\hbox{for all $\varepsilon\in (0,\varepsilon_0)$ }.
	\end{equation}
	Since ${\sf G}$ is a connected semisimple $\Q$-group, the subgroup $\Gamma={\sf G}(\Z)\cap G$ is a lattice in $G$ (see \cite[\S 4.6]{PR}). 
	We apply Proposition \ref{p:upper2} to the action of $\Gamma$ on the space $Y=G/H$, where
	$H={\sf H}(\R)^0$, and the function $|\cdot|$ being given by
	a norm on $\hbox{M}_n(\mathbb{R})$.
	It remains to verify Conditions A1, A2, A3$^\prime$, A4.
	Condition A3$^\prime$ with $\zeta=m$ is immediate from \eqref{eq:e}, and 
	Condition A4 is straightforward to check for norms.
	
	When $H$ is a connected semisimple Lie group,
	the asymptotics of the volume of the norm balls $H_t=\{h\in H;\, \|h\|<t\}$ can be computed explicitly (see \cite[Sec.~7]{GW}, \cite{Mau}, \cite[Sec.~6]{GOS}): 
	$$
	m_H(H_t)\sim c\, t^{b}(\log t)^{b'-1}\quad\hbox{as $t\to \infty$},
	$$
	with $c>0$, $b\in \Q^+$ and $b'\in \N$. 
	The exponent $b$ is positive when $H$ is non-compact
	(equivalently, when ${\sf H}$ is isotropic over $\R$). This verifies Condition A1.

	We also note that the exponent $b$ can be computed explicitly in terms of the root data of $H$ (see \cite[eq.~(1.10)]{GOS}). We will need explicit estimates
	for $b$ in the next section. If we fix a Cartan subgroup $A$ of $H$ with a set of simple roots $\Delta$ and suppose the action of $A$ on $\mathbb{R}^n$ has a unique highest weight
	$\lambda$, then the exponent $b$ can computed as follows.
	We write 
	$$
	\lambda=\prod_{\alpha\in\Delta} \alpha^{m_\alpha}\quad\hbox{and}\quad
 	\rho^2=\prod_{\alpha\in\Delta} \alpha^{n_\alpha},
 	$$
	where $\rho^2$ denotes the product of positive roots. Then
	\begin{equation}
	\label{eq:b}
	b=\max\left\{\frac{n_\alpha}{m_\alpha}:\, \alpha\in\Delta \right\}.
	\end{equation}
	
	To verify Condition A2, we consider the unitary representation $\pi_Z$ of $G$ acting on
	$L_0^2(Z,\mu_Z)$, where $Z=\Gamma\backslash G$. Recall that $\sf H$ is assumed $\Q$-simple, and $H$ is assumed totally non-compact in $G$. It follows from Corollary \ref{cor:spectral gap} that 
	the restriction $\pi_Z|_H$
	is  an $L^p$-integrable representation of $H$.
	Hence, by Theorem \ref{th:mean}, the action of $H$ on 
	$L_0^2(Z,\mu_Z)$ satisfies a quantitative mean ergodic theorem,
	which verifies Condition A2 with some exponent $\theta>0$. 
	Hence, we conclude that Proposition \ref{p:upper2} applies in this setting,
	and yields that for $\kappa$ satisfying 
	\begin{equation}
	\label{eq:kkk}
	\kappa > \frac{m}{2\theta b},
	\end{equation}
	for almost all $g\in G$ and $\varepsilon\in (0,\varepsilon_0(g,\kappa))$, there exists $\gamma\in \Gamma$ such that
	$$
	\gamma^{-1}gH\in \widetilde{S}_\varepsilon H\quad\hbox{and}\quad \|\gamma\|<\varepsilon^{-\kappa}.
	$$
	Since $F$ is invariant under $H$, it follows from \eqref{eq:s} that $\|F(g^{-1} \gamma x_0)-\xi\|<\varepsilon$. Hence, we conclude that
	$x=\gamma x_0\in X^0\cap \mathbb{Z}^n$ gives a solution of 
	$$
	\|F(g^{-1} x)-\xi\|<\varepsilon\quad\hbox{and}\quad \|x\|<n\|x_0\|\, \varepsilon^{-\kappa}.
	$$
	Then for every $\kappa'>\kappa$ and sufficiently small $\varepsilon$, we obtain
	$$
	\|F(g^{-1} x)-\xi\|<\varepsilon\quad\hbox{and}\quad \|x\|<\varepsilon^{-\kappa'}.
	$$
	This completes the proof.
\end{proof}

\begin{proof}[Proof of Theorem \ref{mainthm2}]
	We recall that the stabiliser of the line $[x_0]$ in ${\sf G}$ is assumed to be
	a parabolic $\Q$-subgroup of ${\sf G}$.
	We refer to \cite[\S III.1-2]{bj} for properties of rational parabolic subgroups
	and basic facts from reduction theory.
	Without loss of generality, we may assume that ${\sf P}$
	is a standard rational parabolic subgroup.
	Namely, it is defined as follows.
	Let ${\sf T}$ be a maximal $\Q$-split torus of ${\sf G}$ and
	$\Delta=\Delta({\sf G},{\sf T})$ denote a set of (restricted) simple roots on ${\sf T}$ 
	(for the action of ${\sf T}$ on the Lie algebra of ${\sf G}$).
    For a proper subset $I$
	of $\Delta$, we have ${\sf P}={\sf N}Z_{\sf G}({\sf T}_I),$
	where ${\sf N}$ is the unipotent radical of ${\sf P}$, and 
	${\sf T}_I$ is a connected component of $\cap_{\alpha\in I} \hbox{ker}(\alpha)$.
	For a suitable choice of Cartan involution, there exists a Levi subgroup ${\sf L}$
	defined over $\Q$ which is invariant under this involution.
	We denote by $K$ the maximal compact subgroup of ${\sf G}(\R)$ corresponding to this involution.
	Let ${\sf S}$ be the split center of ${\sf L}$ over $\Q$ and ${\sf M}=\cap_{\chi\in X({\sf L})} \hbox{ker}(\chi^2)$.
	We set
	$$
	N={\sf N}(\R),\quad A={\sf S}(\R)^0, \quad M={\sf M}(\R).
	$$
	Then we have the decompositions
	\begin{equation}
	\label{eq:iwa}
	{\sf P}(\R)=N M A\quad\hbox{and}\quad  {\sf G}(\R)=N M A K.
	\end{equation}
	For $s\in \R$, we define
	$$
	A_s=\{a\in A:\, \alpha(a)>e^s\hbox{ for all $\alpha\in \Delta\backslash I$} \}.
	$$
	For a fixed compact $\Omega\subset  N M^0$ with non-empty interior, we define
	the Siegel sets
	$$
	\Sigma(s)=\Omega A_{s} K^0
	$$
	The Haar measure on ${\sf G}(\R)$ with respect to the decomposition \eqref{eq:iwa} is given by
	$$
	\int_{{\sf G}(\R)} f\, dm_{{\sf G}(\R)} =\int_{N\times M\times A\times K} f(nmak)\, dm_N(n)dm_M(m) \Delta_P(a^{-1})dm_A(a)dm_K(k)
	$$
	for $f\in C_c({\sf G}(\R))$, where 
	$m_*$ denotes the Haar measures on the factors, and 
	$$
	\Delta_P(a)=\det(\hbox{Ad}(a)|_{{\rm \tiny Lie}(N)})=
	\prod_{\alpha\in \Delta\backslash I} \alpha^{m_\alpha}
	$$
	for some positive integers $m_\alpha$.
	This implies that for some $\zeta>0$, we have the bound
	$$
	m_G(\Sigma(s))\gg e^{-\zeta s}.
	$$
	We set $\Gamma={\sf G}(\Z)\cap G$ and $Z=\Gamma\backslash G$,
	and consider the projections of the Siegel sets
	$$
	S_\delta=\Gamma\Sigma(-\log \delta)\subset Z.
	$$
	Since the set $\{\gamma\in \Gamma: \, \Sigma(s)\cap \gamma \Sigma(s)\}$ is finite,
	it follow that the projection map $G\to Z:g\mapsto \Gamma g$ 
	restricted to $\Sigma(s)$ is uniformly finite-to-one.
	This implies that 
	$$
	\mu_Z(S_\delta)\gg \delta^{\zeta}.
	$$
	We apply Proposition \ref{p:upper} to the action of $H={\sf H}(\R)^0$ on the space $Z$.
	Condition A3 for the sets $S_\delta$ has just been verified, and Conditions A1 and A2 have been
	verified in the proof of Theorem \ref{mainthm}. Hence, by Proposition \ref{p:upper},
	when $\kappa>\zeta/(2\theta b)$, for almost all $g\in G$ and $\delta\in (0,\delta_0(g,\kappa))$,
	there exists $h\in H$ such that
	$$
	\Gamma gh\in S_\delta\quad \hbox{and}\quad \|h\|<\delta^{-\kappa}.
	$$
	This implies that there exists $\gamma\in \Gamma$ such that
	$\gamma gh\in \Sigma(-\log \delta)$.
	
	For $g=nma\in P$, we have $g^{-1} x_0=e^{-\lambda(a)}x_0$
	for a character $\lambda$ of $A$.
	Since $[x_0]$ is stabilised by ${\sf P}$, it follows that it is 
	a highest weight vector with respect to the root system $\Delta$.
	In particular, $\lambda$ must be dominant, so that 
	$\lambda=\sum_{\alpha\in \Delta} n_\alpha \alpha$ for some $n_\alpha\ge 0$.
	Hence, it follows that there exists $c>0$ such that
	$$
	\|r^{-1}x_0\|\ll \delta^c\quad\hbox{for $r\in \Sigma(-\log \delta)$. }
	$$
	We conclude that the elements $\gamma$ and $h$ that we have constructed satisfy
	\begin{equation}
	\label{eq:n}
	\|h^{-1}g^{-1}\gamma^{-1}x_0\|\ll \delta^c.
	\end{equation}
	This implies that $x=\gamma^{-1}x_0\in X^0\cap \Z^n$ satisfies 

	\begin{equation}
	\label{eq:g1}
	\|x\|\le n^2 \|g\|\|h\| \|h^{-1}g^{-1}x\|\ll_g \delta^{-(\kappa-c)}.
	\end{equation}
	Let $d:=\min(\deg(F_1),\ldots,\deg(F_m))\ge 1$. Then 
	since the components $F_i$ of the map $F$ are homogeneous,
	it follows from \eqref{eq:n} that
	\begin{equation}
	\label{eq:g2}
	\|F_g(x)\|=\|F(h^{-1}g^{-1}\gamma^{-1}x_0)\|\ll_g \delta^{c d}.
	\end{equation} 
	Hence, combining \eqref{eq:g1} and \eqref{eq:g2}, we conclude that
	when $\kappa>\zeta/(2\theta b)$, for almost all $g\in G$, there exists $c_1,c_2>0$
	such that for all sufficiently small $\delta>0$, the system
	$$
	\|F_g(x)\|< c_1\, \delta^{c d},\quad \|x\|< c_2\, \delta^{-(\kappa-c)}
	$$
	has a solution $x\in X^0\cap \Z^n$. Equivalently, 
	when
	\begin{equation}
	\label{eq:kk2}
	\kappa'=\frac{\kappa-c}{c d}>\frac{\zeta-2\theta bc}{2\theta b c d},
	\end{equation}
	for almost all $g\in G$, $c_3=c_3(g)>0$, and $\varepsilon\in (0,\varepsilon_0(g,\kappa'))$,
	there exists $x\in X^0\cap \Z^n$ satisfying
	$$
	\|F_g(x)\|<\varepsilon\quad\hbox{and}\quad \|x\|<c_2\,\varepsilon^{-\kappa'}.
	$$
	Then for every $\kappa''>\kappa'$ and sufficiently small $\varepsilon$, we obtain
	$$
	\|F(g^{-1} x)-\xi\|<\varepsilon\quad\hbox{and}\quad \|x\|<\varepsilon^{-\kappa''}.
	$$
	This proves the theorem.
\end{proof}

\section{Proof of Theorems \ref{ternary}--\ref{th:gramm}}
\label{sec:proof2}

\begin{proof}[Proof of Theorem \ref{ternary}]
	We observe that this result fits the setting of Theorems \ref{mainthm}--\ref{mainthm2} with ${\sf G}=\hbox{SL}_3$ and ${\sf H}=\hbox{SO}_{Q_0}$
	for a fixed rational indefinite quadratic form $Q_0\in \mathcal{Q}({2,1};\ell)$. 
	
	Since  every $Q\in \mathcal{Q}({2,1};\ell)$ is of the form $Q(x)=Q_0(g^{-1}x)$ for some $g\in \hbox{SL}_3(\R)$, it is sufficient to show that for almost all
	$g\in \hbox{SL}_3(\R)$ and $\varepsilon\in (0,\varepsilon_0(g,\xi,\kappa))$,
	there exists $x\in \Z^3$ such that
	\begin{equation}\label{eq:q_0}
	|Q_0(g^{-1}x)-\xi|<\varepsilon\quad\hbox{and}\quad \|x\|<\varepsilon^{-\kappa}.
	\end{equation}
	We analyse the cases $\xi=0$ and $\xi\ne 0$ separately.
	In both cases, we consider the action of $H=\hbox{SO}_{Q_0}(\R)^0$ on the space $Z=\hbox{SL}_3(\Z)\backslash \hbox{SL}_3(\R)$.
	The corresponding unitary representation of $H$ on $L^2_0(Z)$ is well-known to be tempered (see Section \ref{sec:unitary}), 
	so that the mean ergodic theorem (Condition A2) holds with $\theta=1/2$.
	The volume growth of the norm balls $H_t$ in $H$
	can be computed using the formula \eqref{eq:b}.
	The Cartan subgroup of $H$ is conjugate to $\{c_t=\hbox{diag}(e^t,e^{-t},1):\, t\in \R\}$.
	There is a single positive root $\alpha:c_t\mapsto e^{t}$ and the highest weight 
	$\lambda:c_t\mapsto e^{t}$ is the same,
	so that the volume growth exponent is $b=1$. 
	This verifies Condition A1 with $b=1$.
	
	Suppose that $\xi\ne 0$. Then it is a regular value of the map $Q_0$.
	Hence, Theorem \ref{mainthm} implies that when $\kappa>1/(2\theta b)=1$,
	\eqref{eq:q_0} is solvable for almost all $g$ and all sufficiently small $\varepsilon>0$.
	
	Suppose that $\xi=0$. We apply Theorem \ref{mainthm2} 
	with ${\sf G}=\hbox{SL}_3$ and $x_0=e_1$. The maximal torus ${\sf T}$ of ${\sf G}$
	consists of diagonal matrices in $\sf G$. If we choose the simple roots
	$$
	\alpha_1:a\mapsto a_1a_2^{-1}\quad\hbox{and}\quad \alpha_2:a\mapsto a_2a_3^{-1}\quad\hbox{ with $a=\hbox{diag}(a_1,a_2,a_3)\in {\sf T}.$}
	$$
	Then $x_0$
	is the highest weight vector for the representation of ${\sf G}$ on $\mathbb{C}^3$.
	Its projective stabiliser is the parabolic subgroup 
	${\sf P}=\left(
	\begin{tabular}{ccc}
	$*$ & $*$ & $*$ \\
	$0$ & $*$ & $*$ \\
	$0$ & $*$ & $*$ 
	\end{tabular}
	\right)$,
	and we have the Iwasawa decomposition 
	$$
	{\sf P}(\mathbb{R})=NMA,
	$$ 
	where 
	\begin{align*}
	N&=\left(
	\begin{tabular}{ccc}
	$1$ & $*$ & $*$ \\
	$0$ & $1$ & $0$ \\
	$0$ & $0$ & $1$ 
	\end{tabular}
	\right),\quad
	M=\left(
	\begin{tabular}{ccc}
	$1$ & $0$ & $0$ \\
	$0$ & $*$ & $*$ \\
	$0$ & $*$ & $*$ 
	\end{tabular}
	\right),\\
	A&=\{a_t=\hbox{diag}(e^{2t/3},e^{-t/3},e^{-t/3}):\, t\in\R\}.
	\end{align*}
	Then 
	$$
	\alpha_1(a_t)=e^t,
	$$ so that the Siegel sets $\Sigma(s)$ are given by the 
	condition $\{t>s\}$. 
	We also compute that 
	$$
	\Delta_P(a_t)=e^{2t}.
	$$
	Hence, 
	$$
	m_G(\Sigma(s))\gg e^{-2s},
	$$
	and $\zeta=2$.
	We observe that 
	$$
	a_t\cdot x_0=e^{2t/3}x_0,
	$$
	so that 
	$$
	\|a^{-1}x_0\|\le e^{-2s/3}\quad\hbox{ for $a\in \Sigma(s)$,}
	$$
	and $c=2/3$.
	We apply Theorem \ref{mainthm2} to the polynomial map $Q_0$. We have $d=\deg(Q_0)=2$.
	According to Theorem \ref{mainthm2}, when 
	$$
	\kappa>\frac{\zeta-2\theta bc}{2\theta b c d}=(2-2/3)/(4/3)=1,
	$$
	the system \eqref{eq:q_0} with $\xi=0$ is solvable for almost all $g$ and all sufficiently small $\varepsilon>0$. This completes the proof.
\end{proof}

\begin{proof}[Proof of Theorem \ref{systems}]
	We recall (see \cite[Lemma~2.2]{S1}) that the pair $(Q,F)$ can be reduced to the canonical form 
	\begin{align*}
	Q_0(x_1,\ldots,x_n)&=Q_0'(x_1,\ldots,x_m)+Q_0''(x_{m+1},\ldots,x_n),\\
	F_0(x_1,\ldots,x_n)&=(x_1,\ldots,x_m),
	\end{align*}
	where $Q_0'$ is a non-degenerate form in $m$ variables,
	and $Q_0''$ is a diagonal non-degenerate indefinite form.
	More precisely,
	there exist $g_1\in \hbox{GL}_n(\R)$ and $g_2\in \hbox{GL}_m(\R)$ such that
	$$
	Q(x)=Q_0(g_1^{-1}x)\quad \hbox{and}\quad F(x)=g_2 F_0(g_1^{-1}x).
	$$
	Let ${\sf X}_0=\{Q_0=k\}$. It is easy to check directly that the map $F_0: {\sf X}_0(\R)\backslash\{0\}\to \R^m$ is a surjective submersion.
	The quadratic surface ${\sf X}_0(\R)\backslash\{0\}$ is connected unless it has signature $(1,n-1)$ or
	$(n-1,1)$. In the later case, ${\sf X}_0(\R)\backslash\{0\}$ has two connected components which
	are determined by the sign of one of the coordinates $x_i$.
	Hence, we conclude that for connected components $X_0^{(i)}$ of ${\sf X}_0(\R)\backslash\{0\}$,
	the map $F_0: X_0^{(i)}\to \R^m$ is also surjective submersion.
	Since the map $x\mapsto g_1x$
	defines a diffeomorphism of ${\sf X}_0(\R)\backslash\{0\}$ and ${\sf X}(\R)\backslash\{0\}$,
	it follows that $F: X^{(i)}\to \R^m$ is a surjective submersion
	on connected components $X^{(i)}$ of ${\sf X}(\R)\backslash\{0\}$. 
	
	We demonstrate that this result fits into the framework of Theorem \ref{mainthm}. 
	Let ${\sf G}=\hbox{SO}_Q$ and ${\sf H}$ be the stabilizer of the map $F$ in ${\sf G}$. Then 
	${\sf G}$ and ${\sf H}$ are algebraic $\Q$-groups. The group ${\sf H}$ preserves the linear subspace $V=\{F=0\}$. 
	Since $Q|_{V}$ is non-degenerate, it follows that ${\sf H}$
	preserves the direct-sum orthogonal decomposition $V\oplus V^\perp$.
	Moreover, since the map $F$ defines a coordinate system on $V^\perp$,
	it follows  that ${\sf H}$ acts trivially on $V^\perp$.
	Hence, we have an isomorphism ${\sf H}\simeq \hbox{SO}_{Q|_V}$.
	In particular, it follows that ${\sf H}$ is semisimple and $H={\sf H}(\R)^0\simeq \hbox{SO}(2,1)^0$.
	We set $G={\sf G}(\R)^0$ and $\Gamma={\sf G}(\Z)\cap G$.
	According to Section \ref{sec:quad},
	the representation of $H$ on $L^2_0(\Gamma\backslash G)$ is tempered,
	so that the mean ergodic theorem (Condition A2)
	holds for the action of $H$ on $Z=\Gamma\backslash G$ with the exponent $\theta=1/2$.
	The volume growth of the norm balls in $H\simeq \hbox{SO}(2,1)^0$ is the same as in the proof of Theorem \ref{ternary} and is given by $b=1$. 
	This verifies Condition A1 with $b=1$.
	We conclude that Theorem \ref{mainthm}
	applies to this setting and when $\kappa>m/(2\theta b)=m$, the system
	$$
	\|F_g(x)-\xi\|<\varepsilon,\quad \|x\|<\varepsilon^{-\kappa}\quad \hbox{ with $x\in {\sf X}(\Z)$}
	$$
	is solvable for almost all $g\in G$ and $\varepsilon\in (0,\varepsilon_0(g,\xi,\kappa))$.
	This proves the theorem.
\end{proof}

\begin{proof}[Proof of Theorem \ref{detmap}]
	We observe that the group ${\sf G}=\hbox{SL}_3\times \hbox{SL}_3$ 
	naturally acts on ${\sf X}$ by $x\mapsto g_1 xg_2^{-1}$ for $(g_1,g_2)\in {\sf G}$.
	The action 
	of ${\sf G}(\R)$ on 
	${\sf X}(\R)$ is  transitive.
	We denote by ${\sf H}\simeq \hbox{SL}_3$ the diagonal subgroup in ${\sf G}$.
	Then the polynomial map $F=(F_1,F_2)$ is invariant under ${\sf H}$.
	
	We set 
	\begin{align*}
	G&={\sf G}(\R)=G_1\times G_2=\hbox{SL}_3(\R)\times \hbox{SL}_3(\R),\\
	\Gamma &=\Gamma_1\times \Gamma_2=\hbox{SL}_3(\Z)\times \hbox{SL}_3(\Z),\\
	H&={\sf H}(\R)\simeq \hbox{SL}_3(\R).
	\end{align*}
	Here $H$ is the diagonal subgroup in $G$.
	We define a norm on $G$ using its representation on $\hbox{M}_3(\R)$:
	$x\mapsto g_1xg_2^{-1}$, $g=(g_1,g_2)\in G$. Namely,
	$$
	\|g\|=\max\left\{\|g_1xg_2^{-1}\|:\, \|x\|=1\right\}.
	$$
	The volume growth of the balls $H_t$ defined by this norm 
	can be computed using \eqref{eq:b}.
	A Cartan subgroup of $H\simeq \hbox{SL}_3(\R)$ consists of diagonal matrices
	with simple roots $\alpha_1:a\mapsto a_1a_2^{-1}$ and $\alpha_2:a\mapsto a_2a_3^{-1}=a_1a_2^2$. The product of positive roots is $\rho^2=\alpha_1^2\alpha_2^2$, and the highest weight is
	$\lambda:a\mapsto a_1a_3^{-1}$, so that $\lambda=\alpha_1\alpha_2=\rho$.
	Hence, according to \eqref{eq:b}, the volume growth exponent is given by 
	$b=2$. This verifies Condition A1 with $b=2$.
	
	Given $\xi=(\xi_1,\xi_2)\in \mathbb{R}^2$, we set
	$$
	x_\xi=\left(\begin{tabular}{ccc}
	 0 & 0 & $\ell$ \\
	 1 & 0 & $\xi_1$ \\
	 0 & 1 & $\xi_2$
	\end{tabular}
	\right)\in {\sf X}(\R).
	$$
	Then $F(x_\xi)=(\xi_1,\xi_2)$. In particular, it follows that the map $F:{\sf X}(\R)\to \R^2$ is onto and is a submersion.
		Let us choose $x_0\in \hbox{M}_3(\Z)$ such that $\det(x_0)=\ell$
		and $g_0\in \hbox{SL}_3(\R)$ such that $x_\xi=g_0^{-1} x_0$.
		Then the map $g\mapsto F(g^{-1} x_0)$, $g\in G_1$, is also a submersion at $g_0$.
		Hence, for sufficiently small $\varepsilon>0$, 
		we may define a family of shrinking bounded
		neighbourhoods $O_\varepsilon$ of $g_0$ such that 
		\begin{equation}
		\label{eq:ooo}
		\|F(g^{-1}x_0)-\xi\|<\varepsilon\quad\hbox{for $g\in O_\varepsilon$},
		\end{equation}
		and 
		\begin{equation}
		\label{eq:ooo_o}
		 \varepsilon^2\ll m_{G_1}(O_\varepsilon)\ll \varepsilon^2.
		\end{equation}
		We also fix a compact subset $\Omega_0$ in $\hbox{SL}_3(\R)$ with positive measure.
		Let 
		$$
		\widetilde{S}_\varepsilon=\{(gh,h):\, g\in O_\varepsilon, h\in \Omega_0 \}\subset G \quad
		\hbox{and}\quad S_\varepsilon=\Gamma \widetilde{S}_\varepsilon\subset Z.
		$$
		It is easy to check using \eqref{eq:ooo_o} that 
		$$
		m_G(\widetilde{S}_\varepsilon)\gg \varepsilon^2.
		$$
		In particular, Condition A3$'$ with $\zeta=2$ holds.
		Furthermore, taking $\Omega_0$ and $\varepsilon$ sufficiently small, 
		we may arrange that the factor map $G\to \Gamma\backslash G$
		is one-to-one on $\widetilde{S}_\varepsilon$. In particular, we also have 
		\begin{equation}
		\label{eq:ooo_o2}
		\mu_Z(S_\varepsilon)\gg \varepsilon^2.
		\end{equation}
		Since the map $F$ is ${\sf H}$-invariant, it follows from \eqref{eq:ooo} that 
		\begin{equation}
		\label{eq:ooo2}
		\|F(g_1^{-1}x_0 g_2)-\xi\|<\varepsilon\quad\hbox{for $(g_1,g_2)\in \widetilde{S}_\varepsilon$}.
		\end{equation}

	We consider the action of $H$ on the space $Z=\Gamma \backslash G$ and try to proceed as in the proof of Theorem \ref{mainthm}.
	Note however that in the present case the representation of $H$ on $L^2_0(Z,\mu_Z)$ is only $L^{4+}$-integrable. 
	Indeed $H$ is totally non-compact in $G$ so  the constants are the only $H$-invariant functions in $L^2(Z,\mu_Z)$, and $H\cong SL_3(\R)$ has integrability exponent $4^+$.  
	Then the mean ergodic theorem (Condition~A2) holds with $\theta=1/4$,
	and the previous argument would only imply existence of approximation when $\kappa>\frac{\zeta}{2\theta b}=2$.
	Nonetheless, we now turn to show that it is possible to modify the proof
	of Theorem \ref{mainthm} to produce the desired result and treat all $\kappa>1$.

	Let $\cH=L^2(Z,\mu_Z)$. We consider the unitary representation $\pi_Z$ of $G=G_1\times G_2$ on $\cH$.
	Let $\cH_i$, $i=1,2$, denote the closed $G$-invariant subspace of $\cH$ consisting of the $G_i$-invariant vectors. The unitary representation of $G$ on $(\cH_1+\cH_2)^\perp$
	has no non-zero vectors invariant under either $G_1$ or $G_2$.
	Then it is known that of the representation of the 
	diagonal subgroup $H$ acting on $(\cH_1+\cH_2)^\perp$ is tempered (see Section \ref{sec:unitary}).
	Hence, from Theorem \ref{th:mean} we obtain that for all $\eta>0$,  
	$$
	\|\pi_Z(\beta_t)f\|_{L^2(Z,\mu_Z)}\ll_\eta m_H(H_t)^{-1/2+\eta}\|f\|_{L^2(Z,\mu_Z)},\quad f\in (\cH_1+\cH_2)^\perp.
	$$
	Let $\cH_i^0$ denote the orthogonal complement of constant functions in $\cH_i$.
	The representation of $H$ on $\cH_i^0$ is $L^{4+\eta}$-integrable for all $\eta>0$
	(see Section \ref{sec:unitary}).
	In this case Theorem \ref{th:mean} gives a weaker bound: for all $\eta>0$,  
	$$
	\|\pi_Z(\beta_t)f\|_{L^2(Z,\mu_Z)}\ll_\eta m_H(H_t)^{-1/4+\eta}\|f\|_{L^2(Z,\mu_Z)},\quad f\in \cH_i^0,\;\; \hbox{ with $i=1,2$.}
	$$
	We have the orthogonal $G$-invariant decomposition 
	$$
	\cH=\left<1\right>\oplus \cH_1^0\oplus  \cH_2^0\oplus (\cH_1+\cH_2)^\perp.
	$$
    Let $f_\varepsilon$ denote the characteristic function of 
	 the subset $S_\varepsilon$ of $Z$.
	Combining the above estimates, we obtain that for all $\eta>0$,
	\begin{align*}
	&\norm{\pi_Z(\beta_t)f_\varepsilon-\int_Z f_\varepsilon\, d\mu_{Z}}_{L^2(Z,\mu_Z)}\\
	\ll_\eta &\; m_H(H_t)^{-1/2 +\eta}\, \norm{f_\varepsilon}_{L^2(Z,\mu_Z)}
	\\
	&+
	m_H(H_t)^{-1/4 +\eta}\, \left(\norm{f^{(1)}_\varepsilon}_{L^2(Z,\mu_Z)}+\norm{f^{(2)}_\varepsilon}_{L^2(Z,\mu_Z)}\right),
	\end{align*}
	where $f^{(i)}_\varepsilon$, $i=1,2$, denote the orthogonal projection of $f_\varepsilon$ on the subspaces $\cH_i$. We observe that these projections can be computed explicitly.
	We write 
	$$
	Z=Z_1\times Z_2\quad\hbox{with $Z_i\simeq \Gamma_i\backslash G_i$.}
	$$
	Then $\mu_Z=\mu_{Z_1}\otimes \mu_{Z_2}$, where $\mu_{Z_i}$ denote 
	the invariant probability measures on the factors $Z_i$.
	For $(z_1,z_2)\in Z$,
	\begin{align*}
	f^{(1)}_\varepsilon(z_1,z_2)=\int_{Z_1} f_\varepsilon(z,z_2)\, d\mu_{Z_1}(z)
	=\mu_{Z_1}(\{z\in Z_1:\, (z,z_2)\in S_\varepsilon\}).
	\end{align*}
	Writing $z=\Gamma_1 g$ and $z_2=\Gamma_2 g_2$, we observe that
	$$
	\{(g,g_2)\in G_1\times G_2:\, \Gamma (g,g_2)\in S_\varepsilon\}\subset \Gamma
	\{(g,g_2)\in G_1\times G_2:\, g g_2^{-1}\in O_\varepsilon, g_2\in \Omega_0\}.
	$$
	In particular, we obtain
	\begin{align*}
	\mu_{Z_1}(\{z\in Z_1:\, (z,\Gamma_2 g_2)\in S_\varepsilon\})
	&\ll m_{G_1}(\{g\in G_1:\, g g_2^{-1}\in O_\varepsilon, g_2\in \Omega_0\})\\
	&\le m_{G_1}(O_\varepsilon),
	\end{align*}
	and
	\begin{align*}
	\|f^{(1)}_\varepsilon\|_{L^2(Z,\mu_Z)}&=\left(\int_{Z_2} \mu_{Z_1}(\{z\in Z_1:\, (z,z_2)\in S_\varepsilon\})^2\,d\mu_{Z_2}(z_2)\right)^{1/2}\\
		&\ll m_{G_1}(O_\varepsilon) m_{G_2}(\Omega_0)^{1/2} \ll m_{G_1}(O_\varepsilon).
	\end{align*}
	The $L^2$-norm of the projection $f^{(2)}_\varepsilon$ can be bounded similarly.
	Using these estimates, we conclude that for any $\eta>0$,
	\begin{align*}
	\norm{\pi_Z(\beta_t)f_\varepsilon-\int_Z f_\varepsilon\, d\mu_{Z}}_{L^2(Z,\mu_Z)}\ll_\eta &\, m_H(H_t)^{-1/2 +\eta}\, m_{G_1}(O_\varepsilon)^{1/2}
	\\
	&+
	m_H(H_t)^{-1/4 +\eta}\, m_{G_1}(O_\varepsilon).
	\end{align*}
	Next, we can carry our the argument exactly as in the proof of Proposition \ref{p:upper}.
	Using the estimates \eqref{eq:ooo_o} and \eqref{eq:ooo_o2}, we derive (as in \eqref{eq:zzz}) the bound
	\begin{align*}
	\mu_Z(Z(t,\varepsilon))^{1/2}\ll_\eta  
	&\, \mu_Z(S_\varepsilon)^{-1} \left(m_H(H_t)^{-1/2 +\eta}\, m_{G_1}(O_\varepsilon)^{1/2}
	+
	m_H(H_t)^{-1/4 +\eta}\, m_{G_1}(O_\varepsilon) \right)\\
	\ll  
	&\, m_H(H_t)^{-1/2+\eta}\varepsilon^{-1}+
	m_H(H_t)^{-1/4+\eta}.
	\end{align*}
	Since the volume growth exponent of $H_t$'s is given by $b=2$,
	the proof of Proposition \ref{p:upper} works for $\varepsilon=t^{-1/\kappa}$ with any $\kappa>1$.
\end{proof}

\begin{proof}[Proof of Theorem \ref{th:gramm}]
	We observe that by choosing an integral basis of $\mathbb{R}^3$, the space of real unimodular frame in $\mathbb{R}^3$
	can be identified with the group $G=\hbox{SL}_3(\mathbb{R})$. Then 
	the subset of integral frames is given by the lattice $\Gamma=\hbox{SL}_3(\mathbb{Z})$.
	Let 
	$$
	Q_0(x)=-x_1^2-x_2^2+x_3\quad\hbox{and}\quad J=\hbox{diag}(-1,-1,1).
	$$
	The Gram matrix map for $Q_0$ is given by
	$$
	F_{Q_0}(u)={}^tu J u\quad \hbox{ for $u\in G$.}
	$$
	We claim that the map $F_{Q_0}: G\to \mathcal{Q}(2,1;1)$ is a submersion for all $g\in G$.
	Indeed, since $u\mapsto g_1u$ with $g_1\in G$ defines a diffeomorphism of $G$, 
	it is sufficient to
	verify the claim when $g=e$, which can be checked by a direct computation. 
	Given $\xi\in\hbox{Sym}(2,1;1)$, we take $u_0\in G$ such that $F_{Q_0}(u_0^{-1})=\xi$.
	Since the map $u\mapsto F_{Q_0}(u^{-1})$ is also a submersion, one can construct a collection of shrinking subsets 
	$\tilde{S}_\varepsilon$, $\varepsilon\in(0,\varepsilon_0)$, such that
	\begin{equation}
	\label{eq:v0}
	\|F_{Q_0}(s^{-1})-\xi\|<\varepsilon\quad\hbox{for all $s\in \tilde{S}_\varepsilon$,}
	\end{equation}
	and
	\begin{equation}
	\label{eq:vvv}
	m_G(\tilde{S}_\varepsilon)\gg \varepsilon^\zeta\quad\hbox{for all $\varepsilon\in (0,\varepsilon_0)$,}
	\end{equation}
	where $\zeta=\dim (\hbox{Sym}(2,1;1))=5$.
	Let $H=\hbox{SO}_{Q_0}(\mathbb{R})^0$.
	We apply Proposition \ref{p:upper2} to the action of $\Gamma$ on the space $Y=G/H$.
	In this case $|\cdot|$ is given by the $\max$-norm on $\hbox{M}_3(\mathbb{R})$,
	and $H_t=\{h\in H:\, \|h\|<t\}$ is as in the proof of Theorem \ref{ternary},
	so that the volume growth exponent is $b=1$.
	This verifies Condition A1 with $b=1$.
	Condition A3$'$ with $\zeta=5$ is verified by \eqref{eq:vvv}.
	As we discussed in Section \ref{sec:unitary}, the representation of $H$ on $L_0^2(\Gamma\backslash G)$
	is tempered. Hence, the mean ergodic theorem (Condition A2) holds with $\theta=1/2$.
	We conclude from Proposition \ref{p:upper2} that given any $\kappa>\frac{\zeta}{2\theta b}=5$, for almost all $g\in G$
	and $\varepsilon\in (0,\varepsilon_0(g,\kappa))$, there exists $\gamma\in\Gamma$ satisfying
	$$
	\gamma^{-1}gH\in \tilde S_\varepsilon H\quad\hbox{and}\quad \|\gamma\|<\varepsilon^{-\kappa}.
	$$
	Since ${}^thJh=J$ for $h\in H$, it follows from \eqref{eq:v0} that 
	$$
	\|{}^t(g^{-1}\gamma)J(g^{-1}\gamma)-\xi\|<\varepsilon.
	$$
	Every $Q\in \mathcal{Q}(2,1;1)$ can be represented as $Q(x)=Q_0(g^{-1}x)$ for some $g\in G$.
	Then $Q$ is represented by the matrix
	 $J_Q={}^tg^{-1} Jg^{-1}$, and its Gram matrix map is
	$$
	F_Q(u)={}^t u J_Q u={}^t(g^{-1}u)J(g^{-1}u)
	\quad \hbox{ for $u\in G$.}
	$$
	Hence, we conclude that the Theorem \ref{th:gramm} holds for almost all
	$Q\in \mathcal{Q}(2,1;1)$.
\end{proof}

\section{More on the Pigeonhole Heuristics}
\label{sec:naive}

In this section, we prove Proposition \ref{th:naive} 
which constructs explicit family of polynomial maps
that fail to satisfy the Pigeonhole Heuristics.
Recall that 
$$
{\sf X}=\{x^{2}_1 + \dots + x^{2}_{n-1} - x^{2}_n=1\}
$$
with $n\ge 4$, and 
\begin{align*}
F_{{\alpha}}(x) &=x_n - \sum_{i = 1}^{s}\alpha_i x_i,
\quad\hbox{${\alpha}=(\alpha_1,\ldots,\alpha_s)\in \mathbb{R}^s$.}
\end{align*}
with $1\le s\le n-1$. We start the proof of Proposition \ref{th:naive} with a lemma:
	
\begin{Lemma}\label{l:dioph}
Let $\sigma_s>s-2$ when $s\ge 2$ and $\sigma_1> -1/2$. 
Then for every $\xi\in \mathbb{R}$ and almost all ${\alpha}=(\alpha_1,\ldots,\alpha_s)\in \mathbb{R}^s$, there exists $c=c(\xi,{\alpha})>0$ such that 
$$
\left|z-\left(\sum_{i=1}^s \alpha_i x_i+\xi\right)^2\right|\ge \frac{c}{\|x\|^{\sigma_s}}
$$
for all $z\in \mathbb{Z}$ and ${x}=(x_1,\ldots,x_s)\in\mathbb{Z}^s\backslash \{0\}$
satisfying $z\ge \|{x}\|^2-1$.
\end{Lemma} 

\begin{proof}
Excluding a subset $(\alpha_1,\ldots,\alpha_s)\in\mathbb{R}^s$ of measure zero,
we may assume that 
$$
z-\left(\sum_{i=1}^s \alpha_i x_i+\xi\right)^2\ne 0\quad
\hbox{for all $z\in \mathbb{Z}$ and $(x_1,\ldots,x_s)\in\mathbb{Z}^s\backslash \{0\}$.}
$$
Indeed, when $(x_1,\ldots,x_s)\ne (0,\ldots,0)$,
this inequality defines a complement of at most two hyperplanes in $\mathbb{R}^s$. Hence, it is sufficient to exclude a union of countably many hyperplanes.

Let $\Omega$ be a compact domain in $\mathbb{R}^s$.
For $z\in\mathbb{Z}$ and ${x}\in\mathbb{Z}^s\backslash \{0\}$, we define
$$
A(z,{x})=\left\{{\alpha}\in\Omega:\, \left|z-\left(\sum_{i=1}^s \alpha_i x_i+\xi\right)^2\right|< \frac{1}{\|x\|^{\sigma_s}}
\right\}.
$$
We claim that almost all $(\alpha_1,\ldots,\alpha_s)\in \Omega$
belong only to finitely many of the sets 
$A(z,{x})$ with $z\ge \|{x}\|^2-1$. Once this claim is proved, in view of the previous paragraph,
it will follow that for almost all $(\alpha_1,\ldots,\alpha_s)\in \Omega$,
$$
\min_{z,{x}:z\ge \|{x}\|^2-1} \left|z-\left(\sum_{i=1}^s \alpha_i x_i+\xi\right)^2\right| \|x\|^{\sigma_s}>0.
$$
This immediately implies the lemma for almost every $(\alpha_1,\ldots,\alpha_s)\in \Omega$.
Exhausting $\mathbb{R}^s$ be compact domains, we deduce that
the lemma also holds for almost every $(\alpha_1,\ldots,\alpha_s)\in \mathbb{R}^s$.

Now we proceed with the proof of the claim. We have to show that
the $\limsup$ of the sets
$A(z,{x})$ with $z\in\mathbb{Z}$ and ${x}\in\mathbb{Z}^s\backslash \{0\}$ satisfying $z\ge \|{x}\|^2-1$ has measure zero. 
By the Borel--Cantelli Lemma,
it is sufficient to show that the sum of the volumes of these sets is finite.
We note that given ${x}$, we have $A(z,{x})\ne \emptyset$ for at most two $z$'s.
Hence, it is sufficient to consider ${x}$ with $\|{x}\|\ge 2$.
We use that 
\begin{equation}
\label{eq:vol}
\hbox{vol}(A(z,{x}))\ll \frac{1}{\|{x}\|^{\sigma_s+1}\sqrt{z}}.
\end{equation}
To prove \eqref{eq:vol}, we suppose, for instance, that $|x_1|=\|{x}\|$.
The set $A(z,{x})$ is defined by the inequalities
$$
\sqrt{z-\frac{1}{\|{x}\|^{\sigma_s}}}<\left|\sum_{i=1}^s \alpha_i x_i+\xi\right|<\sqrt{z+\frac{1}{\|{x}\|^{\sigma_s}}}.
$$
This in particular implies that $\alpha_1$ is contained in a pair of
intervals (depending on $\alpha_2,\ldots,\alpha_s$) of length
$$
\frac{1}{|x_1|}\left(\sqrt{z+\frac{1}{\|{x}\|^{\sigma_s}}}-\sqrt{z-\frac{1}{\|{x}\|^{\sigma_s}}}\right) 
\ll
\frac{1}{\|x\|^{\sigma_s+1}\sqrt{z}}.
$$
This bound implies \eqref{eq:vol}.

Suppose that $s=1$. In this case, we have to verify the claim for $\sigma_1>-1/2$,
and without loss of generality we may assume that $\sigma_1<0$.
For every $x_1$, there are at most $O(|x_1|^{-\sigma_s})$
values of $z$ such that $\hbox{vol}(A(z,x_1))\ne \emptyset$.
Hence, we obtain that
$$
\sum_{z,x_1: |x_1|\ge 2, z\ge x_1^2-1} \hbox{vol}(A(z,x_1))\ll
\sum_{x_1\ne 0} \frac{1}{|x_1|^{2\sigma_s+2}} <\infty
$$
since $\sigma_s>-1/2$.

Suppose that $s\ge 2$. In this case, $\sigma_s>0$.
For every ${x}\ne 0$, there are at most two
values of $z$ such that $\hbox{vol}(A(z,{x}))\ne \emptyset$.
As above we obtain
$$
\sum_{z,{x}:\|{x}\|\ge 2, z\ge \|{x}\|^2-1} \hbox{vol}(A(z,{x}))\ll
\sum_{{x}\ne 0} \frac{1}{\|{x}\|^{\sigma_s+2}} <\infty
$$
since $\sigma_s>s-2$.
\end{proof}

\begin{proof}[Proof of Proposition \ref{th:naive}]
We observe that since $\xi\notin \mathbb{Z}$,
the system \eqref{eq:c1} does not have solutions with $(x_1,\ldots,x_s)=(0,\ldots, 0)$
when $\varepsilon$ is sufficiently small.
Hence, it is sufficient to consider solutions satisfying 
$(x_1,\ldots,x_s)\ne (0,\ldots, 0)$.
We shall show that the theorem holds on the set of ${\alpha}$'s of full measure provided by Lemma \ref{l:dioph}. Suppose that 	
$x\in \mathbb{Z}^n$ satisfies \eqref{eq:c1}. Then
$$
x_n=\sum_{i=1}^s \alpha_ix_i+\xi+O(\varepsilon),
$$
and it follows that
$$
x_1^2+\cdots +x_{n-1}^2-\left(\sum_{i=1}^s \alpha_ix_i+\xi\right)^2=1+O(\varepsilon^{1-\kappa_s}).
$$ 
Let $z=x_1^2+\cdots +x_{n-1}^2-1$. Then $z\ge \|(x_1,\ldots,x_s)\|^2-1$, and we have
$$
\left|z-\left(\sum_{i=1}^s \alpha_ix_i+\xi\right)^2\right|\ll \varepsilon^{1-\kappa_s}.
$$
On the other hand,  Lemma \ref{l:dioph} implies the lower bound
$$
\left|z-\left(\sum_{i=1}^s \alpha_ix_i+\xi\right)^2\right|\gg \varepsilon^{\sigma_s\kappa_s}.
$$
Hence, if integral solutions of \eqref{eq:c1} exist for all sufficiently small
$\varepsilon>0$, we must have $1-\kappa_s\le \sigma_s \kappa_s$.
Hence, $\kappa_s\ge 1/(\sigma_s+1)$. We conclude that $\kappa_1\ge 2$,
and  $\kappa_s\ge 1/(s-1)$ when $s>1$. This proves the theorem.
\end{proof}

\end{document}